\documentclass[10pt]{article}
\usepackage{amsfonts,amsmath, amssymb,latexsym,mathrsfs,upgreek}
\usepackage{multirow}
\usepackage{verbatim}
\usepackage{longtable}
\usepackage{hyperref}
\hypersetup{colorlinks=true,linkcolor=blue, linktocpage}
\usepackage{tocloft}
\usepackage{cancel}

\usepackage{tikz}
\usetikzlibrary{matrix,arrows,decorations.pathreplacing}

\usepackage{todonotes}

\usepackage{amsthm}
\usepackage{mathtools}
\usepackage{manfnt}
\usepackage{enumitem}
\usepackage{setspace}
\usepackage{calc,footnote}
\usepackage{graphicx}
\usepackage{tikz}
\usepackage{tikz-cd}
\usepackage{comment}

\newtheorem{Theorem}{Theorem}
\newtheorem{theorem}{Theorem}[section]
\newtheorem{proposition}[theorem]{Proposition}
\newtheorem{lemma}[theorem]{Lemma}
\newtheorem{corollary}[theorem]{Corollary}

\theoremstyle{definition}

\renewenvironment{proof}{\noindent {\bf Proof:}}{\hfill $\Box$ \vspace{2 ex}}

\newtheorem{Remark}[Theorem]{Remark}
\newtheorem*{theorem*}{Theorem}
\newtheorem*{proposition*}{Proposition}
\newtheorem*{lemma*}{Lemma}
\newtheorem{definition}[theorem]{Definition}
\theoremstyle{remark}

\newcommand\nc{\newcommand}
\nc{\on}{\operatorname}
\nc\renc{\renewcommand}
\nc{\BR}{\mathbb R}
\nc{\BC}{\mathbb C}
\nc{\BQ}{\mathbb Q}
\nc{\BZ}{\mathbb Z}
\nc{\BN}{\mathbb N}
\nc{\BP}{\mathbb P}
\nc{\BA}{\mathbb A}
\nc{\Hom}{\on{Hom}}
\nc{\wt}{\widetilde}
\nc{\vspan}{\on{span}}
\nc{\ord}{\on{ord}}
\nc{\im}{\on{im}}
\nc{\Mat}{\on{Mat}}
\nc{\can}{\on{can}}
\nc{\coker}{\on{coker}}
\nc{\ev}{\on{ev}}
\nc{\Tr}{\on{Tr}}
\nc{\End}{\on{End}}
\nc{\swap}{\on{swap}}
\nc{\Set}{\on{Set}}
\nc{\bC}{{\bm C}}
\nc{\bc}{{\bm c}}
\nc{\bD}{{\bm D}}
\nc{\bd}{{\bm d}}
\nc{\bE}{{\bm E}}
\nc{\be}{{\bm e}}
\nc{\bff}{{\bm f}}
\nc{\Bf}{\mathbb{f}}

\nc{\adj}{\on{adj}}
\nc{\tensor}[3]{#1 \underset{#2}\otimes #3}
\nc{\Nat}{\on{Nat}}
\nc{\op}{\on{op}}
\nc{\funct}{\on{funct}}
\nc{\Ob}{\on{Ob}}
\nc{\fR}{\mathfrak{R}}
\nc{\Vect}{\on{Vect}}
\nc{\ns}{\on{non-spec}}
\nc{\ol}{\overline}
\nc{\ul}{\underline}
\nc{\w}{\omega}
\nc{\nlog}{\on{nlog}}
\nc{\Aut}{\on{Aut}}
\nc{\Gal}{\on{Gal}}
\nc{\Poss}{\on{Poss}}
\nc{\ksep}{\on{sep}}
\nc{\low}{\on{low}}
\nc{\Stab}{\on{Stab}}
\nc{\pp}{\mathfrak{p}}
\nc{\OO}{\mathcal{O}}
\nc{\mm}{\mathfrak{m}}
\nc{\qq}{\mathfrak{q}}
\nc{\Nm}{\on{Nm}}
\nc{\Ann}{\on{Ann}}
\nc{\gug}{\mathfrak{g}}
\nc{\hug}{\mathfrak{h}}
\nc{\mf}{\mathfrak}
\nc{\mc}{\mathcal}
\nc{\Sym}{\on{Sym}}
\renc{\O}{\mc{O}}
\nc{\al}{\alpha}

\def\Z{{\mathbb Z}}
\def\I{{\mathbb I}}
\def\End{{\rm End}}
\def\Sym{{\rm Sym}}

\def\SL{{\rm SL}}
\def\GL{{\rm GL}}

\def\Stab{{\rm Stab}}
\def\Sym{{\rm Sym}}

\def\Cl{{\rm Cl}}
\def\O{{\mathcal O}}
\def\cO{{\mathcal O}}
\def\SO{{\rm SO}}

\def\P{{\mathbb P}}

\def\fin{{\rm fin}}

\def\adj{{\rm adj}}
\def\gb{{\rm gb}}
\def\Aut{{\rm Aut}}

\def\gen{{\rm gen}}
\def\ngen{{\rm ngen}}
\def\rr{{\rm rr}}
\def\red{{\rm red}}

\def\red{{\rm red}}
\def\prim{{\rm prim}}

\def\Vol{{\rm Vol}}

\def\R{{\mathbb R}}
\def\F{{\mathbb F}}
\def\FF{{\mathcal F}}

\def\Q{{\mathbb Q}}

\def\cB{{\mathcal B}}

\def\cI{{\mathcal I}}
\def\I{{\mathcal I}}

\def\Z{{\mathbb Z}}
\def\P{{\mathbb P}}
\def\F{{\mathbb F}}
\def\Q{{\mathbb Q}}
\def\C{{\mathbb C}}

\def\Res{{\rm{Res}}}

\makeatletter
\def\@tocline#1#2#3#4#5#6#7{\relax
  \ifnum #1>\c@tocdepth
  \else
    \par \addpenalty\@secpenalty\addvspace{#2}
    \begingroup \hyphenpenalty\@M
    \@ifempty{#4}{
      \@tempdima\csname r@tocindent\number#1\endcsname\relax
    }{
      \@tempdima#4\relax
    }
    \parindent\z@ \leftskip#3\relax \advance\leftskip\@tempdima\relax
    \rightskip\@pnumwidth plus4em \parfillskip-\@pnumwidth
    #5\leavevmode\hskip-\@tempdima
      \ifcase #1
       \or\or \hskip 1em \or \hskip 2em \else \hskip 3em \fi
      #6\nobreak\relax
    \dotfill\hbox to\@pnumwidth{\@tocpagenum{#7}}\par
    \nobreak
    \endgroup
  \fi}
\makeatother

 \allowdisplaybreaks

\begin{document}

\title{Secondary terms in the counting functions of quartic fields II}

\author{Arul Shankar and Jacob Tsimerman}

\maketitle

\abstract{We determine the smoothed counts of $S_4$-quartic fields with bounded discriminant, satisfying any finite specified set of local conditions, as the sum of two main terms with a power saving error term. We also prove an analogous result for quartic rings (weighted by the number of cubic resolvents), deducing as a consequence that the Shintani zeta functions associated to the prehomogeneous vector space $\C^2\otimes\Sym^2(\C^3)$ have at most a simple pole at~$s=5/6$.}

\section{Introduction}

In the first article \cite{ST_second_main_term_1} of this two-part series, we carried out the ``smooth count'' of quartic $S_4$-fields, in certain natural families, with two main terms of sizes $X$ and $X^{5/6}$. Our main result there was restricted to what we called {\it $S_4$-families}, which are families of quartic fields, satisfying a finite set of local conditions, which ensures that every field in the family is a quartic $S_4$-field. The purpose of this article is to remove this restriction, and provide a smooth count for families of quartic fields satisfying any finite set of local conditions.\footnote{In our main result in this article, we evaluate smoothed counts of $S_4$-fields in {\it all} congruence families of quartic fields. Since asymptotics for non-$S_4$-quartic fields are already known (with an error term of $o(X^{5/6})$), our result can be used to deduce smoothed counts for the family of all quartic fields in any congruence family with two main terms.}

To state our main result, we need some notation. Let $S$ be a finite set of places, and let $\Sigma=(\Sigma_v)_{v\in S}$ be a {\it finite collection of local specifications for quartic fields}, where for each $v\in S$, $\Sigma_v$ is a finite set of \'etale quartic extensions of $\Q_v$. We let $F(\Sigma)$ denote the family of $S_4$-quartic fields $K$ such that $K\otimes\Q_v\in\Sigma_v$ for every $v\in S$.
Let $\psi:\R_{\geq 0}\to \R_{\geq 0}$ be a smooth function with compact support, and define the ``smooth count'' of quartic fields in $F(\Sigma)$ by
\begin{equation}\label{eq:fields_count}
N_{\Sigma}(\psi,X):=\sum_{K\in F(\Sigma)}\psi\Bigl(\frac{|\Delta(K)|}{X}\Bigr).
\end{equation}
For a principal ideal domain $R$, Bhargava's landmark work \cite{BHCL3} gives a parametrization of the set of quartic rings $Q$ over $R$ along with a cubic resolvent ring $C$ of $Q$. Given $Q$ and $C$, there is a natural quadratic resolvent map $r:Q/R\to C/R$. In this paper we work with triples $(Q,C,r)$. Given such a triple over $R$, and given an element $x\in (C/R)^\vee$, we obtain (by composition with $r$) a quadratic form $Q/R\to R$, which we denote by $r_x$. Specialize to the case $R=\Q_p$ for some $p$. For an \'etale quartic extension $K$ of $\Q_p$, Bhargava proves that the ring of integers $\O_K$ of $K$ has a unique cubic resolvent $C_K$. We say that $(\O_K,C_K,r_K)$ is the {\it triple corresponding to $K$}, we denote the quadratic form corresponding to an element $x$ in $(C_K/\Z_p)^\vee$ by $r_{K,x}$, and the set of primitive elements in $(C_K/Z_p)^\vee$ by $(C_K/\Z_p)^\vee_\prim$. Define the constants
\begin{equation}\label{eq:M_def}
\mathcal M:= \frac{2^{5/3}\Gamma(1/6)\Gamma(1/2)}{\sqrt{3}\pi\Gamma(2/3)};\;\;
\mathcal M_i:=\left\{
\begin{array}{rcl}
\mathcal M&{\rm if}&i\in\{0,2\};
\\
\sqrt{3}\cdot\mathcal M&{\rm if}&
i=1;
\end{array}
\right.
\;\mathcal M_i':=\left\{
\begin{array}{rcl}
\mathcal M&{\rm if }& i=0;\\
\sqrt{3}\cdot \mathcal M&{\rm if }& i=1;\\
\displaystyle\frac{\mathcal M}{3}&{\rm if }& i=2.
\end{array}
\right.
\end{equation}
Then we have the following result generalizing \cite[Theorem 2]{ST_second_main_term_1}.

\begin{Theorem}\label{th_main_counting_result}
Let $\Sigma$ be any finite collection of local specifications and let $F(\Sigma)$ be the corresponding family of $S_4$-quartic fields.
Let $\psi:\R_{\geq 0}\to\R_{\geq 0}$ be a smooth function with compact support. Then 
\begin{equation*}
N_{\Sigma}(\psi,X)=C_1(\Sigma)\wt{\psi}(1)\cdot X+ C_{5/6}(\Sigma)\wt{\psi}(5/6)\cdot X^{5/6} + O(X^{13/16+o(1)}),
\end{equation*}
where
\begin{equation}\label{eq:leadingcons}
C_1(\Sigma):=\frac{1}{2}\Bigl(\sum_{K\in\Sigma_\infty}\frac{1}{\#\Aut(K)}\Bigr)\prod_p
\Bigl(\sum_{K\in\Sigma_p}\frac{|\Delta(K)|_p}{\#\Aut(K)}\Bigr)\Bigl(1-\frac{1}{p}\Bigr),
\end{equation} and

\begin{align*}    
C_{5/6}(\Sigma)&:=\frac{\pi}{8}\Bigl(\mathcal M_{\Sigma}\cdot\zeta(1/3)\prod_{p}\Bigl(1-\frac{1}{p^{1/3}}\Bigr)\sum_{K\in\Sigma_p}\frac{|\Delta(K)|_p}{\#\Aut(K)}\cdot \int_{x\in (C_K/\Z_p)^{\vee}_{\prim}}  |\det{r_{K,x}}|_p^{-2/3}dx \\&\;\;\;+ \;\mathcal M_\Sigma'\cdot\zeta(2/3)\prod_{p}\Bigl(1-\frac{1}{p^{2/3}}\Bigr)\sum_{K\in\Sigma_p}\frac{|\Delta(K)|_p}{\#\Aut(K)}\cdot \int_{x\in (C_K/\Z_p)^{\vee}_{\prim}}  \epsilon_p(r_x)|\det{r_{K,x}}|_p^{-2/3}dx\Bigr).
\end{align*}
Above, $(\cO_K,C_K,r_K)$ is the triple corresponding to $K$, $\epsilon_p$ denotes the Hasse invariant, and 
\begin{equation*}
\mathcal M_\Sigma = \sum_{K=\R^{4-2i}\times\C^i\in\Sigma_\infty}\frac{\mathcal M_i}{|\Aut(K)|};
\quad\quad
\mathcal M_\Sigma' = 
\sum_{K=\R^{4-2i}\times\C^i\in\Sigma_\infty}\frac{\mathcal M_i'}{|\Aut(K)|}.
\end{equation*}
\end{Theorem}
\noindent See \cite[\S 9.6]{ST_second_main_term_1} for the explicit values in the product for various $\Sigma$.

Cohen--Diaz y Diaz--Olivier \cite[Corollary 6.1]{Cohen_Diaz_Olivier} prove that the number of $D_4$-quartic fields, with discriminant bounded by $X$, grows like a constant times $X$ up to an error of $O(X^{3/4+o(1)})$. Durlanik \cite{MD_thesis} proves versions of this result, where the count is smoothed, and the family of all $D_4$-quartic fields is replaced by a family satisfying finitely many local conditions. Meanwhile, the number of $V_4$-, $C_4$-, and $C_2\times C_2$-fields with discriminant less than $X$ is bounded by $O(X^{1/2+o(1)})$ by work of Baily \cite{Baily}, and the number of $A_4$-quartic fields with discriminant less than $X$ is bounded by $O(X^{.778...})$ in \cite[Theorem 1.4]{BSTTTZ}. Therefore, Theorem \ref{th_main_counting_result} can be used to obtain the smooth count of all quartic fields (or indeed, quartic fields satisfying any finite set of local conditions) with bounded discriminant.

\medskip

In \cite[Theorem 3]{ST_second_main_term_1}, we prove a smoothed counting results for $S_4$-families of quartic rings (along with their resolvents). In this article, we generalize that result to any congruence family. We now precisely define the set of families that we can handle: $\Lambda$ is said to be a {\it finite collection of local specifications for quartic rings} if $\Lambda=(\Lambda_v)_{v\in S}$, where for all $v$ in a finite set of places $S$, the set $\Lambda_v$ consists of a finite set of quartic triples over $\Z_v$ with non-zero discriminant. We define $R(\Lambda)$ to be the set of quartic triples $(Q,C,r)$ over $\Z$, whose base change to $\Z_v$ lies in $\Lambda_v$ for every $v\in S$, and such that $Q$ is an order in an $S_4$-quartic field. Let $\psi:\R_{\geq 0}\to\R$ be a smooth function with compact support. We define the smoothed count of quartic triples in $R(\Lambda)$ analogously to \eqref{eq:fields_count}:
\begin{equation}\label{eq:rings_count}
N_{\Lambda}(\psi,X):=\sum_{(Q,C,r)\in R(\Lambda)}\psi\Bigl(\frac{|\Delta(Q)|}{X}\Bigr).
\end{equation}
Then we have the following result:

\begin{Theorem}\label{thm:mainS4rings}
Let $\Lambda$ be any finite collection of local specifications for quartic rings and let $R(\Lambda)$ be the corresponding family of quartic triples. Let $\psi:\R_{\geq 0}\to\R_{\geq 0}$ be a smooth function with compact support. Then 
\begin{equation*}
N_{\Lambda}(\psi,X)=C_1(\Lambda)\wt{\psi}(1)\cdot X+ C_{5/6}(\Lambda)\wt{\psi}(5/6)\cdot X^{5/6} + O(X^{3/4}\log X),
\end{equation*}
where
\begin{equation*}\label{eq:leadingcons}
C_1(\Lambda)=\frac{1}{2}\Bigl(\sum_{K\in\Lambda_\infty}\frac{1}{\#\Aut(K)}\Bigr)\prod_p
\Bigl(\sum_{(Q,C,r)\in\Lambda_p}\frac{|\Delta(Q)|_p}{\#\Aut((Q,C,r)}\Bigr)\Bigl(1-\frac{1}{p}\Bigr),
\end{equation*} and
\begin{align*}\label{eq:secondcons}    
C_{5/6}(\Lambda)&=\frac{\pi}{8}\Bigl(\mathcal M_\Lambda\zeta(1/3)\prod_{p}\Bigl(1-\frac{1}{p^{1/3}}\Bigr)\cdot\sum_{(Q,C,r)\in\Lambda_p}\frac{|\Delta(Q)|_p}{\#\Aut(Q,C,r)}\cdot \int_{x\in (C/\Z_p)^{\vee}_{\prim}}  |\det{r_x}|_p^{-2/3}dx \\&\;\;\;+\; \mathcal M_\Lambda'\zeta(2/3)\prod_{p}\Bigl(1-\frac{1}{p^{2/3}}\Bigr)\cdot\sum_{(Q,C,r)\in\Lambda_p}\frac{|\Delta(Q)|_p}{\#\Aut(Q,C,r)}\cdot \int_{x\in (C/\Z_p)^{\vee}_{\prim}}  \epsilon_p(r_x)|\det{r_x}|_p^{-2/3}dx\Bigr).
\end{align*}
Above, notation is as in Theorem \ref{th_main_counting_result}.
\end{Theorem}

For a ring $R$ we define $V(R)=R^2\otimes \Sym^2(R^3)$. The group $\GL_2\times \GL_3$ acts naturally on $V$, and this representation is prehomogeneous (the action over $\C$ has a unique open orbit). Foundational work of Sato--Shintani \cite{SatoShintani} associates Shintani zeta functions (defined in \S2.3) to prehomogeneous representations, and they develop general theory proving that these representations have meromorphic continuation to the complex plane with constrained orders and locations of the poles. For the specific case of $(V,G)$, it follows that the corresponding Shintani zeta functions have possible poles of order at most four at $s=1$, and of order at most two at $s=5/6$ and $s=3/4$. From Bhargava's landmark counting results \cite{dodqf}, it follows that the possible pole at $s=1$ has order at most $2$. Using our counting results, we prove an analogous result about the possible pole at $s=5/6$.

\begin{Theorem}\label{th_shintani_residue_main}
Let $\phi:V(\Z)\to\R$ be a $G(\Z)$-invariant function, defined by congruence conditions modulo a positive integer. Then for $i\in\{0,1,2\}$, the Shintani zeta functions $\xi_i(\phi;s)$ $($defined in \S2.3$)$ have at most a simple pole at $s=5/6$.
\end{Theorem}

Before describing how the above results are proved, we make the following remark.
\begin{Remark}\label{Rem_main_thms}
{\rm The poles of the zeta functions $\xi_i(\phi;s)$ are closely related to the counting functions of quartic triples, since quartic triples are parametrized, by foundational work of Bhargava \cite{BHCL3}, by $G(\Z)$-orbits on $V(\Z)$. However, for the purpose of computing the poles and residues of $\xi_i(\phi;s)$, it is necessary to determine the power series expansion of the smooth counts of {\em all } quartic triples, not merely $S_4$ triples as carried out in Theorem \ref{thm:mainS4rings}. In \S3, we give a recipe to count quartic triples $(Q,C,r)$, where $Q$ is an order in a non $S_4$ \'etale quartic extension of $\Q$. Determining the power series expansion of this count following that recipe, and combining it with Theorem \ref{thm:mainS4rings} will formally recover the residues of the poles of $\xi_i(\phi;s)$. 
}\end{Remark}

\subsection*{Outline of the proof}

The methods we use to prove our main results are to a large extent complementary to the methods in the first article of our series \cite{ST_second_main_term_1}. We begin by describing how Theorems \ref{thm:mainS4rings} and \ref{th_shintani_residue_main}, which are closely related, are proved. As explained in Remark \ref{Rem_main_thms}, analogues of these results, for $S_4$-families, are proved in \cite{ST_second_main_term_1}. Our proofs here begin at the same point as there, namely, these results follows from sufficiently precise smooth counts of $G(\Z)$-orbits on $V(\Z)$, where these orbits are weighted by some congruence function $\phi$. These counts can be expressed as an integral which looks like
\begin{equation}\label{eq:FF_int}
\int_{g\in\FF} \#(L\cap g\cB)^{\Delta<X}dg,
\end{equation}
where $\FF$ is a fundamental domain $\FF$ for the action of $\GL_2(\Z)\times\SL_3(\Z)$ on $\GL_2(\R)\times\SL_3(\R)$ and $L\subset V(\Z)$ is a lattice. We note that both Sato--Shintani's and Bhargava's counting methods begin with needing to evaluate such an integral.

Shintani's method of evaluation, carried out  successfully in \cite{MR289428} for the prehomogeneous representation $\Sym^3(2)$ of $\GL_2$ begins by using Poisson summation to estimate the integrand. Simplifying and generalizing Davenport's approach \cite{MR43822}, Bhargava uses geometry-of-numbers tools to evaluate such integrals. Both these approaches have their advantages: Poisson summation gives a very precise estimate of the integrand, from which lower-order terms may be recovered. Meanwhile, geometry-of-numbers methods are highly flexible, allowing for manual removal of undesirable points from $L$ (such as points corresponding to non $S_4$-rings, which typically lie high in the cusp). However, Poisson summation lacks this flexibility\footnote{Indeed, Yukie's impressive work \cite{Yukie} adapting Shintani's technique to the case $(G,V)$ falls short of determining the orders of the poles and their residues.}, and insists on counting all points in $L$, while geometry-of-numbers techniques are unsuited for recovering lower-order terms except in rather simple cases (see \cite{BST,ShankarTaniguchi} for examples). As we will describe in the remainder of the section, this paper demonstrates that our new counting method (introduced in \cite{ST_second_main_term_1}) combines the advantages of both previous methods - precision and flexibility.

\medskip

We divide the integral \eqref{eq:FF_int} over $\FF$ into three regions: first, the {\it round body}, consisting of $g\in\FF$ such that all the coordinates in $g\cB$ are large (growing at least as some power of $X$). Second, the {\it shallow cusp}, consisting of those $g\in \FF$ such that $g$ is not in the round body, but such that $g\cB\cap V(\Z)$ contains at least one element $x$ corresponding to an $S_4$- or $A_4$-ring. We call such elements $x$ {\it generic}. Third, the {\it deep cusp}, consisting of those elements $g\in\FF$ such that $g\cB\cap V(\Z)$ contains only non-generic elements. In \cite{ST_second_main_term_1}, we assumed that $L$ is an {\it $S_4$-lattice}, i.e., a lattice containing only generic elements, and developed a counting method which gave precise estimates for the integral \eqref{eq:FF_int} in the round body and the shallow cusp. This was sufficient since (by the assumption on $L$) the contribution from the deep cusp was $0$. This strategy falls short for our purposes because this time the deep cusp does have a nontrivial (in fact, a dominant!) contribution, and also because separating generic from non-generic elements in our count (necessary to prove Theorems \ref{th_main_counting_result} and \ref{thm:mainS4rings}) is no longer automatic - all our lattices will contain both types of elements. We develop a new set of tools to handle both these issues. More precisely, our strategy to prove Theorem \ref{th_shintani_residue_main} is the following:
\begin{itemize}
\item[{\rm (1)}] We evaluate the analogue of \eqref{eq:FF_int}, with $L$ replaced with the set of non-generic elements in $L$, without employing point-counting methods. We do this by carrying out a precise count of \'etale quartic algebras over $\Q$, which are neither $S_4$- nor $A_4$-quartic fields, and counting subrings inside them (along with resolvent cubic rings). This method gives a power series expansion for the integral, with terms $X\log X$, $X$, and $X^{5/6}$, up to a power saving error term. This is carried out in \S3.
\item[{\rm (2)}] We isolate certain ``special slices'' of $V(\R)$ which have the property of containing only non-generic elements, and account for the majority of the non-generic elements. We note that there do exist non-generic elements outside these slices, which we show constitute a contribution asymptotic to $X^{5/6}$ . Every integral element in the deep cusp lies in one of these slices. (Though these slices also have non-trivial intersection with the main body and shallow cusp.)
\item[{\rm (3)}] We give a power series expansion (with terms $X$ and $X^{5/6}$) for the analogue of \eqref{eq:FF_int}, where we exclude points in the special slices from the count. For this, as previously noted, it suffices to restrict the integral to the main body and shallow cusp. It is only for this step, we use our (suitably modified) counting method from \cite{ST_second_main_term_1}. This is carried out in \S5.
\item[{\rm (4)}] We determine asymptotics for the analogue of \eqref{eq:FF_int}, where we restrict the count to only include
non-generic elements outside these slices. Since the order of growth for this is $\asymp X^{5/6}$, only primary term asymptotics are needed, and this can be done using traditional geometry-of-numbers methods. This is carried out in \S4.
\item[{\rm (5)}] We put everything together: Adding the result of Steps 1 and 3 over-counts \eqref{eq:FF_int} exactly by the result of Step 4. Since there is no $X^{5/6}\log X$ term in any of these three steps, Theorem \ref{th_shintani_residue_main} follows.
\end{itemize}

Once Theorem \ref{th_shintani_residue_main} has been obtained, two tasks remain to prove Theorem \ref{thm:mainS4rings}. First, we must separate the contributions of generic and non-generic rings from the count. And second, we must determine the leading constants of the $X$ and $X^{5/6}$ terms which arise in the counting. The first task is simple: indeed, the contribution of the generic rings count comes by taking the result of Step 3 above, and subtracting from it the result of Step 4. The second task is more complicated because there are two $X^{5/6}$ terms, coming from Steps 3 and 4. The leading coefficients of both of these terms are difficult to compute: The term in Step 3 is some sort of weighted volume inside a 10-dimensional slice in $V(\R)$ (with $a_{11}=b_{11}=0$ to be precise), while the term in Step 4 is a sum of weighted volumes inside a growing number of various different special slices. 

Though we are not really able to evaluate either term precisely, we prove that they must exactly cancel! We do this by showing that as we input a congruence condition of non-maximality at $p$, the total constant changes by a factor of $1/p+O(1/p^2)$. Then we input a result bounding the residues at $5/6$ of the associated Shintani zeta functions from \cite{ST_second_main_term_1}, showing that the rate of decay must be faster. This proves these two terms cancel exactly.
Once this is established, the remain terms (all from Step 3) are evaluated using results from \cite{ST_second_main_term_1}, yielding Theorem \ref{thm:mainS4rings}.

\medskip

Finally, we turn to the proof of Theorem \ref{th_main_counting_result}. For this, we once again input a crucial ingredient from \cite{ST_second_main_term_1}, namely, that the smooth count of all \'etale quartic algebras over $\Q$ has a power series expansion, and that the leading constants of the terms in this expansion can be described as limits of the residues of a sequence of Shintani zeta functions.\footnote{In \cite{ST_second_main_term_1}, we had no way to evaluate the residues, and so (aside from the case of $S_4$-families) we could not evaluate the leading constants of the power saving expansion, proving only that some power series expansion exists. In particular, we could not there prove that the $X^{5/6}\log X$ term does not appear in the power saving expansion for general families.} For our purposes then, we are left with two tasks: first, we must evaluate the limits of these residues (the limit essentially is as more and more congruence conditions of maximality are imposed), and second, we must isolate the $S_4$-contribution from the count of all \'etale quartic algebras. To this end, we begin by separating the contribution from $S_4$-rings and non $S_4$-rings to the residues of the arising Shintani zeta functions. The contribution of $S_4$-rings to the leading constants can be computed via a process similar to the proof of Theorem \ref{thm:mainS4rings}. While the computation of the contribution of non $S_4$-rings to the leading constants can be done in principle, it is in fact not necessary for us! Instead, we simply prove that the limit of these contributions is equal to the respective leading constants of the smooth counts of \'etale quartic non $S_4$-rings. Finally, explicitly taking the limits of the $S_4$-rings part of the leading constants, yields Theorem \ref{th_main_counting_result}.

\section{Background and preliminary results}

In this section we collect some preliminary background and results needed for the rest of the paper.

\subsection{The parametrization of quartic rings (along with their cubic resolvent rings)}

Let $G$ be the algebraic group
\begin{equation}\label{eq:G_definition}
G:=\{(g_2,g_3)\in\GL_2\times\GL_3:\det(g_2)\det(g_3)=1\}.
\end{equation}
We let $G(\R)^+$ denote the set of elements $g=(g_2,g_3)$ with $\det(g_2)>0$. For an element $g\in G(\R)^+$ we define $\lambda(g)=\det(g_3)^{1/6}$.
Let $V$ be the space of pairs of ternary quadratic forms; for a ring $R$, we represent elements $(A,B)\in V(R)$ by pairs of symmetric $3\times 3$ matrices:
\begin{equation*}
(A,B)=\left[\left(
\begin{array}{ccc}
a_{11} & \frac{a_{12}}{2} & \frac{a_{13}}{2}\\[.05in]
\frac{a_{12}}{2} & a_{22} & \frac{a_{23}}{2}\\[.05in]
\frac{a_{13}}{2} & \frac{a_{23}}{2} & a_{33}
\end{array}
\right),
\left(
\begin{array}{ccc}
b_{11} & \frac{b_{12}}{2} & \frac{b_{13}}{2}\\[.05in]
\frac{b_{12}}{2} & b_{22} & \frac{b_{23}}{2}\\[.05in]
\frac{b_{13}}{2} & \frac{b_{23}}{2} & b_{33}
\end{array}
\right)
\right],
\end{equation*}
where $a_{ij}$ and $b_{ij}\in R$. We will consider $a_{ij}$ and $b_{ij}$ to be functions from $V(R)$ to $R$ in the obvious way.
We obtain an action of $G$ on $V$ by restricting the natural action of $\GL_2\times\GL_3$ on $V$:
\begin{equation*}
(\gamma_2,\gamma_3)\cdot(A,B)=\Bigl(
\ \gamma_3A\gamma_3^t, \gamma_3B\gamma_3^t\Bigr)\gamma_2^t.
\end{equation*}
It is well known that the representation $V$ of $G$ is {\it prehomogeneous}, with ring of relative invariants generated by $\Delta\in\Z[V]$, where $\Delta(A,B)$ is defined to be the discriminant of the resolvent binary cubic form:
\begin{equation*}
\Delta(A,B):=\Delta(\Res(A,B)):=\Delta(4\det(Ax-By)).
\end{equation*}
As a consequence of being prehomogeneous, the set $V(\R)^{\Delta\neq }0$ of elements in $V(\R)$ with nonzero discriminant breaks up into finitely many $G(\R)$-orbits: namely, $V(\R)^{(0)}$, $V(\R)^{(1)}$, and $V(\R)^{(2)}$. Specifically, $V(\R)^{(i)}$ consists of elements $(A,B)\in V(\R)^{\Delta\neq 0}$ such that the conics $A$ and $B$ intersect in $4-2i$ real points in $\P^2(\R)$. Given any subset $S$ of $V(\R)$, we let $S^{\Delta\neq 0}$ and $S^{(i)}$ denote $S\cap V(\R)^{\Delta\neq 0}$ and $V(\R)^{(i)}$, respectively.

The following result is due to Bhargava \cite{BHCL3} in the case $R=\Z$ and Wood \cite{WoodThesis,MR2948473} for the case when $R$ is a PID. In fact, Wood's generalization is vastly more general, with versions of the below result when $R$ is replaced with an arbitrary base scheme.

\begin{theorem}\label{th:quartic_param}
Let $R$ be a principal ideal domain. There is a natural bijection between isomorphism classes of triples $(Q,C,r)$, where $Q$ is a quartic ring and $C$ is a cubic resolvent ring of $Q$, and $G(R)$-orbits on $V(R)$. Moreover,if $(Q,C,r)$ corresponds to  $(A,B)$, then $C$ corresponds to the $\GL_2(R)$-orbit of $\Res(A,B)$ under the Delone--Faddeev parametrization. We also have $\Delta(Q)=\Delta(C)=\Delta(A,B)$.
\end{theorem}

Let $(Q,C,r)$ be a triple corresponding to an element $(A,B)\in V(\Z)$ with nonzero discriminant. Then $K=Q\otimes\Q$ has the following choices:
\begin{enumerate}
\item $K$ is an $S_4$- or $A_4$-quartic field;
\item $K$ is the sum of two quadratic fields, or a $D_4$-, $C_4$-, or $V_4$-quartic field;
\item $\Q$ is a summand of $K$.
\end{enumerate}
If we are in Case (a), we will say that $(Q,C,r)$ and $(A,B)$ are {\it generic}; in Case (b), we say that $(Q,C,r)$ and $(A,B)$ have {\it reducible resolvent}; in Case (c), we say that $(Q,C,r)$ and $(A,B)$ are {\it reducible}. Denote the set of elements in $V(\Z)$ that are generic, have reducible resolvent, and are reducible, by $V(\Z)^\gen$, $V(\Z)^\rr$, and $V(\Z)^\red$, respectively. Denote the set of elements in $V(\Z)$ having nonzero discriminant and are not generic by $V(\Z)^\ngen$. Note that these above defined subsets of $V(\Z)$ are all $G(\Z)$-invariant.

It will also be useful for us to define these additional subsets of $V(\Z)$. First, we define the {\it $11$-slice} to be the set of elements in $V(\Z)$ with $a_{11}=b_{11}=0$, and denote this set by $V(\Z)^{a_{11}=b_{11}=0}$. Second, we define the {\it $\det$-slice} to be the set of elements in $V(\Z)$ with $\det(A)=0$, and denote this set by $V(\Z)^{\det(A)=0}$. We define the {\it generic body} to be the complement of these two slices in $V(\Z)^{\Delta\neq 0}$, and denote it by $V(\Z)^{\gb}$. These new subsets are not $G(\Z)$-invariant. So we cannot, for example, talk about counting $G(\Z)$-orbits on $V(\Z)^{\gb}$. However, these sets will be useful when we want to partition integral points in various regions of $V(\R)$ into points we want to count precisely, and points we want to ignore.

We end with the following result.
\begin{proposition}
An element $(A,B)\in V(\Z)$ with nonzero discriminant is reducible if and only if $A$ and $B$ have a common rational zero as conics in $\P^2(\Q)$. An element $(A,B)\in V(\Z)$ with nonzero discriminant has reducible resolvent if and only if $\Res(A,B)$ is reducible over $\Q$.
\end{proposition}
\begin{proof}
Let $(Q,C,R)$ denote the triple corresponding to $(A,B)$. It is well known that $Q\otimes\Q$ is the ring of global sections of the scheme cut out by the quadratic forms $A$ and $B$ in $\P^2_\Q$. This has a summand of $\Q$ in it if and only if the conics $A$ and $B$ have a common rational $0$. Similarly $\Res(A,B)$ is reducible over $\Q$ if and only if $C\otimes\Q$ has a summand of $\Q$ in it. However, $C\otimes\Q$ is the cubic resolvent algebra of $Q\otimes \Q$. It is easy to check that $Q\otimes\Q$ has a reducible resolvent precisely in Case (b) above.
\end{proof}

\noindent Consequently, elements in the $11$-slice are reducible, and elements in the $\det$-slice have reducible resolvents.


\subsection{Coordinates, fundamental domains, and measures}

We will typically express elements in $G(\R)$ in their Iwasawa coordinates. We have
\begin{equation*}
G(\R)=\Lambda NAK,
\end{equation*}
where $N$ is the subgroup of pairs of unipotent lower triangular matrices, $A$ is the subgroup of pairs of diagonal matrices with determinant $1$, $K=\SO_2(\R)\times\SO_3(\R)$ is a maximal compact subgroup of $G(\R)$, and $\Lambda\cong\R^\times$ is the group of elements $(\lambda_2,\lambda_3)$, where $\lambda_2$ is the $2\times 2$ diagonal matrix with $\lambda^{-3}$ as it's coefficients, and $\lambda_3=$ is the $3\times 3$ diagonal matrix with $\lambda^2$ as its coefficients, for $\lambda\in\R^\times$. It is easy to see that $(\lambda_2,\lambda_3)\in\R^\times$ acts on elements in $V(\R)$ by scalar multiplication by $\lambda$, and so we will denote elements in $\Lambda$ simply by $\lambda\in\R^\times$. We write elements in $A$ as $s=(t,s_1,s_2)$, where the $2\times 2$ matrix corresponding to $s$ has $t^{-1}$ and $t$ as its diagonal coefficients while the $3\times 3$ matrix corresponding to $s$ has $s_1^{-2}s_2^{-1}$, $s_1s_2^{-1}$, and $s_1s_2^2$ as its diagonal coefficients. In these coordinates, 
\begin{equation*}
dg=t^{-2}s_1^{-6}s_2^{-6}d^\times\lambda dud^\times sdk
\end{equation*}
is a Haar-measure on $G(\R)$, where $du$ is Haar-measure on $N(\R)$ normalized so that $N(\Z)$ has covolume-$1$ in $N(\R)$, $dk$ is Haar-measure on $K$ normalized so that $K$ has volume $1$, and $d^\times\theta$ denotes $\theta^{-1}d\theta$ for any~$\theta$.
This gives rise to a natural measure for $\SL_2$ and $\SL_3$ as well. We refer to these as $dg_2$ and $dg_3$.

For a smooth and connected group scheme $H/\Z$, we define $\omega_H$ to be a (unique up to sign) top degree left-invariant differential form over $\Z$. For $R=\R$ or $\Z_p$ we denote the corresponding measures on $H(R)$ by $\nu_H$.
This gives rise to Haar measures $\nu_G$, $\nu_{\SL_2}$, and $\nu_{\SL_3}$.
The two Haar measures we have defined on $G(\R)$, $\SL_2(\R)$, and $\SL_3(\R)$ differ by a nonzero constant.
Following \cite[Equation (6)]{ST_second_main_term_1}, we write
\begin{equation}\label{eq:J_def}
J:=\frac{\nu_G}{dg}=6\cdot\frac{\nu_{\SL_2}}{dg_2}\cdot\frac{\nu_{\SL_3}}{dg_3}.
\end{equation}
As in \cite{ST_second_main_term_1}, it will not be necessary for us to compute the value of $J$.

\subsection{Shintani zeta functions, global zeta integrals, and counts}

In this subsection, we collect results from \cite{Kimura_book} on Shintani zeta functions, global integrals, and their relation to various counting functions of interest to us. We will initially work in a more general setting than usual. Specifically, let $\phi:V(\Z)\to\R$ be any bounded function, and let $\psi,\eta:\R_{\geq 0}\to\R_{\geq 0}$ be smooth function with compact support. Fix $i\in\{0,1,2\}$, and let $\cB:V(\R)^{(i)}\to\R$ be a smooth function with compact support away from the discriminant $0$ locus. Then we define the ``zeta integral''
\begin{equation}\label{eq:global_zeta_int_V}
Z(\phi,\cB;s):=\int_{g\in \FF}(\lambda(g))^{-12s}\Bigl(\sum_{x\in V(\Z)}\phi(x)(g\cB)(x)\Bigr)\nu_G(g).
\end{equation}
This zeta integral converges absolutely to the right of $\Re(s)=1$ for $\phi=1$, and hence also for every bounded $\phi$. We define the associated counting function
\begin{equation}\label{eq:I-def}
\cI_\eta(\phi,\cB;X):=\int_{g\in\FF}\sum_{x\in V(\Z)}\phi(x)(g\cB)(x)\eta\Big(\frac{\lambda(g)}{X^{1/12}}\Big)dg.
\end{equation}
A standard application of Mellin inversion relates these two quantities as follows:
\begin{equation}\label{eq:I_to_Z_Mellin}
\frac{J}{12}\cI_\eta(\phi,\cB;X) = 
\int_{2}Z(\phi,\cB;s)\widetilde{\eta}(12s)X^sds,
\end{equation}
where we use the notation $\int_2$ to mean $(2\pi i)^{-1}\int_{\Re(s)=2}$.

When $\phi$ is $G(\Z)$-invariant, we may define for $i\in\{0,1,2\}$, the counting function
\begin{equation*}
N^{(i)}_\psi(\phi;X) := \sum_{x\in\frac{V(\Z)^{(i)}}{G(\Z)}}
\frac{\phi(x)}{|\Stab_{G(\Z)}(x)|}\psi\Bigl(
\frac{|\Delta(x)|}{X}
\Bigr),
\end{equation*}
and its associated Dirichlet series
\begin{equation*}
\begin{array}{rcl}
\xi_{i}(\phi;s)&:=&\displaystyle\sum_{x\in G(\Z)\backslash V(\Z)^{(i)}}\frac{\phi(x)}{|\Stab_{G(\Z)}(x)||\Delta(x)|^{s}}.
\end{array}
\end{equation*}
The function $\xi_i(\phi;s)$ absolutely converges to the right of $\Re(s)=1$ (again, this follows since the fact is known for $\phi=1$). Another application of Mellin inversion connects the above two quantities:
\begin{equation}\label{eq:N_to_xi_Mellin}
N^{(i)}_\psi(\phi;X) = \int_{2}\xi_i(\phi;s)X^s\widetilde{\psi}(s)ds.
\end{equation}

When $\phi$ is a {\it periodic function}, i.e., defined via congruence conditions modulo some positive integer, $\xi_i(\phi;s)$ is the Shintani zeta function associated to $\phi$, and $Z(\phi,\cB;s)$ is called the global zeta integral (associated to $\phi$ and $\cB$). In that case, it is known that $\xi_i(\phi;s)$ and $Z(\phi,\cB;s)$ have meromorphic continuation to the complex plane, with possible poles at most at $1$ (or order at most 4), $5/6$, and $3/4$ (of order at most $2$). 
We note that \eqref{eq:N_to_xi_Mellin} and \eqref{eq:I_to_Z_Mellin} respectively imply that power saving expansions for $\cI_\eta(\phi,\cB;X)$ and $N^{(i)}_{\psi}(\phi;X)$, up to errors of $O(X^{5/6-\theta})$ for any positive $\theta$, determine the orders of the poles of $\xi_i(\phi;s)$ and $Z(\phi,\cB;s)$ at $s=1$ and $5/6$, as well as the residues.

Let $A_i=\#\Aut(\R^{4-2i}\times\C^i)$ be the size of the stabilizer in $G(\R)$ of any element in $V(\R)^{(i)}$. The following result relates $\xi_i(\phi;s)$ and $Z(\phi,\cB;s)$ for any $G(\Z)$-invariant bounded function $\phi$.
\begin{proposition}\label{prop_global_zeta_shintani_G}
With notation as above, we have
\begin{equation*}
Z(\phi,\cB;s)=A_i\xi_i(\phi,s)\int_{x\in V(\R)^{(i)}}|\Delta(x)|^{s-1}\cB(x)\nu_V(x).
\end{equation*}
\end{proposition}
\noindent The above result is proved in \cite[Proposition 5.14]{Kimura_book}. Indeed, while the result is implicitly stated for periodic functions $\phi$, the proof carries over without change for any bounded $G(\Z)$-invariant function.

Next, we use Proposition \ref{prop_global_zeta_shintani_G} to relate the counting functions $N^{(i)}_{\psi}(\phi;X)$ and $I_{\eta}(\phi,\cB;s)$ for some carefully chosen $\psi,\eta$ and $\cB$. To that end, we pick $y\in V(\R)^{(i)}$ such that $|\Delta(y)|=1$ and take smooth functions with compact support $\eta,\eta_1\in C_c^{\infty}(\R_{>0})$ and $\mathcal{G}:G_1(\R)\to\R$, where $G_1(\R)=\ker\lambda$.  We define $\cB:V(\R)^{(i)}\to\R$ by setting 
\begin{equation}\label{eq:good_cB}
\cB(x)=\sum_{\substack{g=\lambda g_1\in G(\R)\\g y=x}}\mathcal{G}(g_1)\eta_1(\lambda).
\end{equation}
We define a smooth, compactly supported function $\psi:\R_{\geq0}\to \R$ such that
\begin{equation}
\wt{\psi}(s)=\wt{\eta}(12s)\wt{\eta_1}(12s).
\end{equation}
From the measure change formulas of the previous subsection, we obtain
\begin{equation*}
\begin{array}{rcl}
\displaystyle\int_{x\in V(\R)^{(i)}}|\Delta(x)|^{s-1}\cB(x)\nu_V(x)&=&\displaystyle J\int_{\lambda}\lambda^{12 s}\eta_1(\lambda)d^\times\lambda\int_{G_1(\R)}\mathcal G(g_1)dg_1
\\[.2in]&=&\displaystyle 
C_{\mathcal{G}}J\widetilde{\eta_1}(12s),
\end{array}
\end{equation*}
for some constant $C_{\mathcal G}$. Then we have the following result.

\begin{proposition}\label{prop:I_to_N}
Let $\phi:V(\Z)\to\R$ be a $G(\Z)$-invariant function. Let $\cB:V(\R)^{(i)}\to\R$ and $\eta,\eta_1,\psi:\R_{\geq 0}\to\R$ be as above. Then we have
\begin{equation*}
I_{\eta}(\phi,\cB;X)=\frac{C_{\mathcal G}A_i}{12}N_{\psi}^{(i)}(\phi;X),
\end{equation*}
where $C_{\mathcal G}$ is the constant above.
\end{proposition}

\begin{proof}
We have
\begin{equation*}
\begin{array}{rcl}
I_{\eta}(\phi,\cB;X)&=&\displaystyle\frac{12}{J}\int_2Z(\phi,\cB;s)\widetilde{\eta}(12s)X^sds
\\[.2in] &=&\displaystyle
\frac{C_{\mathcal G}A_i}{12}\int_2\xi_i(\phi;s)\widetilde{\eta}(12s)\widetilde{\eta_1}(12s)X^sds
\\[.2in] &=&\displaystyle
\frac{C_{\mathcal G}A_i}{12}\int_2\xi_i(\phi;s)\widetilde{\psi}(s)X^sds
\\[.15in]&=&\displaystyle\frac{C_{\mathcal G}A_i}{12}N_{\psi}^{(i)}(\phi;X),
\end{array}
\end{equation*}
as desired.
\end{proof}


\section{Counting reducible rings}

In this section, we count the number of reducible quartic rings, along with their cubic resolvent rings, having bounded discriminant and satisfying local conditions at finitely many places. Let $S$ be a finite set of finite primes. For a prime $p\in S$, we let $\Lambda_p$ be any nonempty set of pairs $(R_p,C_p)$, where $R_p$ is a nondegenerate quartic ring over $\Z_p$, and $C_p$ is a cubic resolvent ring of $R_p$. (Note that this is a more general setup than in the introduction, where we assumed that the set $\Lambda_p$ was finite.) Let $R(\Lambda)$ denote the set of pairs $(Q,C)$, where $Q$ is a quartic ring, $C$ is a cubic resolvent of $R$, and for all $p\in S$, we have $(Q\otimes\Z_p,C\otimes\Z_p)\in\Lambda_p$. Let $\phi_\Lambda$ denote the characteristic function of the set of non-generic pairs $(A,B)$ corresponding to a pair $(Q,C)\in R(\Lambda)$. 

Let $q$ be a positive squarefree integer such that $p\nmid q$ for every $p\in S$. Let $n_q:V(\Z)\to\R$ be the characteristic function of the set of elements $x\in V(\Z)$ which are non-maximal at every prime dividing $q$.
Then we prove the following result.

\begin{theorem}\label{thm_red_points_count}
Let notation be as above, and let $\phi_{q,\Lambda}:V(\Z)\to\R$ be $\phi_\Lambda\cdot n_q$.
Let $\psi:\R_{\geq 0}\to\R$ be a smooth and compactly supported function. Then for $i\in\{0,1,2\}$, we have
\begin{equation*}
N^{(i)}_\psi(\phi_{q,\Lambda};X)=C^{\red '}_{1,q}X\log X+C^\red_{1,q}X+C^\red_{5/6,q}X^{5/6}+O_q(X^{3/4+o(1)}),
\end{equation*}
for some constants $C^{\red '}_{1,q}=C^{\red '}_{1,q}(\Lambda,\psi),\,C^\red_{1,q}=C^\red_{1,q}(\Lambda,\psi),\,C^\red_{5/6,q}=C^\red_{5/6,q}(\Lambda,\psi)$. Moreover, the sums 
\begin{equation*}
\sum_{q}|C^{\red '}_{1,q}|,\quad \sum_{q}|C^\red_{1,q}|,\quad \sum_{q}|C^\red_{5/6,q}|,
\end{equation*}
over squarefree $q$ of these leading constants converge absolutely.
\end{theorem}

Combining the above with Proposition \ref{prop:I_to_N}, we obtain the following immediate consequence.
\begin{corollary}\label{cor:counting_red_rings_I}
Let notation be as above, and let $\cB:V(\R)^{(i)}\to\R$ be as in \eqref{eq:good_cB}. Then we have
\begin{equation*}
\I_\eta(\phi_{q,\Lambda},\cB;X)=c^{\red '}_{1,q}X\log X+c^\red_{1,q}X+c^\red_{5/6,q}X^{5/6}+O_q(X^{3/4+o(1)}),
\end{equation*}
for some constants $c^{\red '}_{1,q}$, $c^\red_{1,q}$, and $c^\red_{5/6,q}$, which satisfy the same ``decay at $q$'' properties as the analogous constants in Theorem \ref{thm_red_points_count}.
\end{corollary}


\subsection{Counting \'etale reducible quartic $\Q$-algebras}

\subsubsection*{Quadratic fields} Counting quartic $\Q$-algebras of the form $\Q\times\Q\times K$ is equivalent to counting quadratic fields $K$, which is what we do next.

\begin{lemma}\label{lem:quad_field_count}
Let $n$ be a squarefree positive integer, and let $\sigma_n=(\sigma_p)_{p\mid n}$ be a collection of quadratic splitting types, one for each prime $p\mid n$. Let $F$ denote the family of quadratic fields $K$ such that every prime $p$ dividing $n$ has splitting type $\sigma_p$ in $K$. Let $\psi:\R_{>0}\to\R_{\geq 0}$ be a compactly supported function. Then
\begin{equation*}
\sum_{\substack{K\in F\\\pm\Delta(K)>0}}\psi\Bigl(\frac{|\Delta(K)|}{X}\Bigr)=c_2(\sigma(n))X+O(n^{1/4+o(1)}X^{1/2}),
\end{equation*}
for some constant $c_2(\sigma(n))\ll_\psi 1$.
\end{lemma}
\begin{proof}
Let $n_0$ be the product of primes dividing $n$ at which the prescribed splitting type is ramification. Let $q$ be a square-free integer
with $(q,n)=1$. We set $S_q$ to be the set of $D$ satisfy the following two conditions:
\begin{enumerate}
    \item The ring $\Z[\sqrt{D}]$ is maximal at all prime $p\mid n$ with the correct splitting type. (This is a congruence condition modulo $2nn_0$.)
    \item We have $q^2|D$. 
\end{enumerate}
There is a bijection between $S_1\backslash (\cup_{q>1} S_q)$ and $F$, where $D$ corresponds to the field $\Q(\sqrt{D})$. Hence, we may compute the sum in the left hand side of the displayed equation lemma by sieving. Note that any $D$ in $S_q$ can be written as $D=q^2n_0D'$, and the congruence condition at any $p\mid n$ on $D'$ picks out either the set of quadratic residues or non-residues. We record the corresponding congruence conditions on $D'$ as $\theta_n(D')$.
Recall that the Fourier transform $\hat{\theta}_nt)$ is bounded by $n^{\frac12+o(1)}$
We therefore obtain:
\begin{equation}\label{eq:inc_exc_quad}
\begin{array}{rcl}
\displaystyle\sum_{\substack{K\in F\\\pm\Delta(K)>0}}\psi\Bigl(\frac{|\Delta(K)|}{X}\Bigr)&=&\displaystyle\sum_{(q,n)=1} \mu(q)\sum_{\substack{D\in S_q\\\pm D>0}}\psi\Bigl(\frac{|D|}{X}\Bigr)
\\[.2in]&=&\displaystyle
\sum_{(q,n)=1} \mu(q)\sum_{\pm D'\in\Z_{>0}} \psi\Bigl(\frac{|D'|}{X/(n_0q^2)}\Bigr) \theta_n(D')
\\[.2in]&=&\displaystyle
\sum_{(q,n)=1} \mu(q)\frac{X}{nn_0q^2}\sum_{\pm t\in\Z_{>0}} \hat{\psi}\Bigl(\frac{tX}{nn_0q^2}\Bigr) \hat{\theta}_n(t).
\end{array}
\end{equation}
The sum ov r$q$ of the contribution from $t=0$ gives the main term. The terms with $t\neq 0$ are regarded as error. It follows that the error for each $q$ is the minimum of $O(X/n_0q^2)$ and $O(n^{1/2 +o(1)})$. Optimizing and summing over $q$ we get an error of $O(X^{1/2}n^{1/4+\epsilon}/n_0^{1/2})$.  Hence, we obtain 
$$\sum_{\substack{K\in F\\\pm\Delta(K)>0}}\psi\Bigl(\frac{|\Delta(K)|}{X}\Bigr)=X\hat{\psi}(0)\cdot\frac{1}{nn_0}\prod_{(p,n)=1}\Bigl(1-\frac{\hat{\theta_p}(0)}{p^2}\Bigr)  + O\Bigl(X^{1/2}n^{1/4+o(1)}\Bigr),$$
as desired.
\end{proof}

\noindent We remark that the above lemma could also be easily derived using $L$-function methods, following for example Wright \cite{Wright}. The $L$-functions which arise from this method are $L(s,\chi_n)$ for quadratic characters modulo $n$, which have conductors $n$. Applying the convex bound on the central values of these $L$-functions would recover our result. The result can be improved by additionally using known subconvexity results, but we do not need these improvements for our purposes.

\subsubsection*{Cubic fields}

Counting quartic $\Q$-algebras of the form $\Q\times K$, for a cubic field $K$, is equivalent to counting cubic fields, which the following result accomplishes.

\begin{theorem}\label{thm:cubic_field_count}
Let $n$ be a squarefree positive integer, and let $\sigma_n=(\sigma_p)_{p\mid n}$ be a collection of cubic splitting types, one for each prime $p\mid n$. Let $F$ denote the family of cubic fields $K$ such that every prime $p$ dividing $n$ has splitting type $\sigma_p$ in $K$. Let $\psi:\R\to\R_{\geq 0}$ be a compactly supported function. Then
\begin{equation*}
\sum_{\substack{K\in F\\\pm\Delta(K)>0}}\psi\Bigl(\frac{|\Delta(K)|}{X}\Bigr)=c_3(\sigma(n))X+c_3'(\sigma(n))X^{5/6}+O(n^{2/3}X^{2/3+o(1)}),
\end{equation*}
for some constants $c_3(\sigma(n)),c_3'(\sigma(n))\ll_\psi 1$.
\end{theorem}
The sharp version of this count has been carried out by Bhargava--Taniguchi--Thorne in \cite[Theorem 1.3]{BTT}. Their methods adapt without significant change (and considerable simplifications!) to the smooth count. Note that we are only prescribing splitting types at $p\mid n$, and not finer local conditions. This is important since otherwise the exponent of $n$ in the error term rises to $n^{8/3}$.

\subsubsection*{Quartic $D_4$-fields}

\begin{theorem}\label{thm:D4-fields}
Let $n$ be a squarefree positive integer, and let $\sigma_n=(\sigma_p)_{p\mid n}$ be a collection of quartic-$D_4$ splitting types, one for each prime $p\mid n$. Let $F$ denote the family of $D_4$-quartic fields $K$ such that every prime $p$ dividing $n$ has splitting type $\sigma_p$ in $K$. Let $\psi:\R\to\R_{\geq 0}$ be a compactly supported function. Then
\begin{equation*}
\sum_{\substack{K\in F\\\pm\Delta(K)>0}}\psi\Bigl(\frac{|\Delta(K)|}{X}\Bigr)=c_{D_4}(\sigma(n))X+O(n^{2/5}X^{3/5+o(1)}),
\end{equation*}
for some constant $c_{D_4}(\sigma(n))\ll_\psi 1$.
\end{theorem}
\begin{proof}
A sharp version of this result, in the case when $n=1$ is proved by Cohen--Diaz y Diaz--Olivier in \cite{Cohen_Diaz_Olivier}. We explain how their proof goes, and what modifications are needed to obtain our result. Every quartic-$D_4$ number field is a quadratic extension of a quadratic extension. Moreover, the number of quadratic extensions of quadratic fields with discriminant less than $X$, which are not $D_4$-quartic fields, is bounded by $O(X^{1/2+\epsilon})$. So for the purpose of proving the result, we may assume that we are counting quadratic extensions $K_4$ of quadratic fields $K_2$. The splitting conditions $\sigma_n$ on $K_4$ can be translated into a finite union (of size $O(n^{o(1)})$) of pairs $(\sigma_n^{(2)},\sigma_n^{(2,2)})$, where $\sigma_n^{(2)}$ is a collection of quadratic splitting conditions, one for each prime dividing $n$, and $\sigma_n^{(2,2)}$ is a collection of splitting types, two or one for each prime $p\in\Z$ dividing $n$, depending on whether $p$ splits or not in the specifications of $\sigma_n^{(2)}$. These splitting types specify the splitting behavior in $K_4$ of each prime factor of $(p)$ in $K_2$. Then we may write
\begin{equation}\label{eq:quartic_D4_prelim}
\sum_{K_2\in F(\sigma_n^{(2)})}\sum_{K_4\in F_{K_2}(\sigma_n^{(2,2)})}
\psi\Bigl(\frac{|\Delta(K_4)|}{X}\Bigr)=
\sum_{K_2\in F(\sigma_n^{(2)})}\sum_{K_4\in F_{K_2}(\sigma_n^{(2,2)})}
\psi\Bigl(\frac{N_{K_2/\Q}\Delta(K_4/K_2)}{X/\Delta(K_2)^2}\Bigr),
\end{equation}
where $F_{K_2}(\sigma_n^{(2,2)})$ denotes the set of quadratic extensions $K_4$ of $K_2$ satisfying the splitting conditions of $\sigma_n^{(2,2)}$. 

We will fix some $Q$ to be optimized later, evaluate the inner sum in the right hand side of the above equation precisely for $|\Delta(K_2)|<Q$, and use a tail estimate for $|\Delta(K_2)|\geq Q$. 
Applying \cite[Lemma 8.1]{D4preprint}, we see that there are $O(N^{o(1)})$ $D_4$-fields with conductor $N$, which implies that there are $O(D^{(o(1)})$ $D_4$-fields with discriminant $D$. We hence obtain
\begin{equation}\label{eq:D4_count_unif}
\sum_{\substack{K_2\in F(\sigma_n^{(2)})\\|\Delta(K_2)|\geq Q}}\sum_{K_4\in F_K(\sigma_n^{(2,2)})}
\psi\Bigl(\frac{N_{K_2/\Q}\Delta(K_4/K_2)}{X/\Delta(K_2)^2}\Bigr)
\ll \frac{X^{1+o(1)}}{Q}.
\end{equation}

To estimate the inner sum in the RHS of \eqref{eq:quartic_D4_prelim}when $|\Delta(K_2)|<Q$, we use the approach of \cite{Cohen_Diaz_Olivier}. They prove in \cite[Theorem 1.1]{Cohen_Diaz_Olivier} that for a fixed field $K_2$, if we denote the set of all quadratic extensions of $K_2$ by $F_{K_2}$, then we have
\begin{equation*}
\sum_{K_4\in F_{K_2}}\frac{1}{N_{K_2/\Q}\Delta(K_4/K_2)^s}=-1+
\frac{1}{\zeta_{K_2}(2s)}\sum_{c^2}\frac{N(2/c)}{N(2/c)^{2s}}\sum_\chi L_{K_2}(s,\chi),
\end{equation*}
where $c$ runs over all integral ideals of $K_2$ dividing $2$, $\chi$ runs over all quadratic characters of the ray class group $\Cl_{c^2}(K_2)$ modulo $c^2$, and $L_{K_2}(s,\chi)$ is the Hecke $L$-function of $K$ for the character $\chi$.
In our situation, the sum is instead over $K_4\in F_{K_2}(\sigma_n^{(2,2)})$. This situation is handled by Durlanik in \cite[Theorem 5.8]{MD_thesis} (see also the discussion immediately following the proof of \cite[Corollary 59]{MD_thesis}), who proves that
\begin{equation*}
\sum_{K_4\in F_{K_2}(\sigma_n^{(2,2)})}\frac{1}{N_{K_2/\Q}\Delta(K_4/K_2)^s}=\delta_\sigma+
\frac{1}{\wt\zeta_{K_2}(2s)}\sum_\chi w_\chi \wt{L}_{K_2}(s,\chi),
\end{equation*}
where $\delta_\sigma$ is a constant bounded by $O((n|\Delta(K_2)|)^{o(1)})$, $\wt\zeta(s)$ (resp.\ $\wt{L}_{K_2}(s,\chi)$) denotes the Dedekind zeta function (resp.\ the Hecke $L$-function of $K_2$ for the character $\chi$) of $K_2$ with suitably modified Euler factors at primes dividing $n$ which are forced to ramify, and $\chi$ runs over quadratic characters of the ray class group of $K_2$ modulo some integral ideal dividing $4n$. Furthermore, the weights $w_\chi$ are absolutely bounded.
We note the two main differences from the Cohen--Diaz y Diaz--Olivier case: first, the length of the sum over $\chi$ is $O(|n\Delta(K_2)|^{o(1)})$ instead of $O(|\Delta(K_2)|^{o(1)})$; second, and more seriously, the conductors of the $L$-functions are of size around $n^2|\Delta(K_2)|$ rather than $|\Delta(K)|$.

The proof of Theorem \ref{thm:D4-fields} is now standard. Expressing the inner sum in the right hand side of~\eqref{eq:quartic_D4_prelim} as a Mellin integral, picking up the (simple) pole at $1$ from the trivial character, and pulling back the integral to $\Re(s)=1/2+\epsilon$, we obtain
\begin{equation*}
\sum_{\substack{K_2\in F(\sigma_n^{(2)})\\|\Delta(K_2)|< Q}}\sum_{K_4\in F_{K_2}(\sigma_n^{(2,2)})}
\psi\Bigl(\frac{N_{K_2/\Q}\Delta(K_4/K_2)}{X/\Delta(K_2)^2}\Bigr)
=
c_{K_2}\frac{X}{\Delta(K_2)^2} + E(K_2)
\end{equation*}
for some constant $c_{K_2}$, where the error term $E(K_2)$ can be bounded using just the convexity bound on $L_{K_2}(s,\chi)$ as follows: 
\begin{equation*}
E(K_2)\ll \frac{X}{|\Delta(K_2)|^2}^{1/2+o(1)}(n^2|\Delta(K_2)|)^{1/4+o(1)}\ll X^{1/2+o(1)}|\Delta(K_2)|^{-3/4+o(1)}n^{1/2+o(1)}.
\end{equation*}
Finally, we sum over $K_2$ in $F(\sigma_n^{(2)})$ with $|\Delta(K_2)|< Q$. The sum of $c_{K_2}/|\Delta(K_2)|^2$ converges; this follows from the fact that it converges when there are no congruence conditions imposed on the pair $(K_2,K_4)$. Meanwhile, the sum of the error terms is bounded by $n^{1/2}Q^{1/4}X^{1/2+o(1)}$.

Optimizing, we pick $Q=X^{2/5}/n^{2/5}$ and obtain the asymptotic with the claimed error term. The final claim regarding $c_{D_4}(\sigma(n))$ is immediate since this constant is upper bounded by the asymptotic constant for the count of all $D_4$-fields.
\end{proof}

\subsubsection*{\'Etale $\Q$-algebras $K\times K'$}

\begin{proposition}\label{prop:22_count}
Let $n$ be a squarefree positive integer, and let $\sigma_n=(\sigma_p)_{p\mid n}$ be a collection of quartic splitting types compatible with a product of quadratic algebras, one for each prime $p\mid n$. Let $F$ denote the family of products of two quadratic algebras $K=K_1\times K_2$ such that every prime $p$ dividing $n$ has splitting type $\sigma_p$ in $K$. Let $\psi:\R\to\R_{\geq 0}$ be a compactly supported function. Then
\begin{equation*}
\sum_{\substack{K\in F\\\pm\Delta(K)>0}}\psi\Bigl(\frac{|\Delta(K)|}{X}\Bigr)=c_{22}'(\sigma(n))X\log X+c_{22}(\sigma(n))X+O(n^{1/2}X^{1/2+o(1)})
+O(n^{1/4}X^{3/4+o(1)}),
\end{equation*}
for some constants $c_{22}'(\sigma(n))\ll_\psi 1$ and $c_{22}(\sigma(n))\ll_\psi n^{o(1)}$.
\end{proposition}

\begin{proof}
An algebra $K_1\times K_2$, where $K_1$ and $K_2$ are different quadratic fields has discriminant $\Delta(K_2)\Delta(K_2)$. We count pairs $(K_1,K_2)$ (ignoring the situation when they are the same algebra, since the number of those is negligible) using a smoothed version of the Dirichlet hyperbola method. Specifically, the collection $\sigma_n$ can be written as a finite union (of size bounded by $n^{o(1)}$) of a pairs of quadratic splitting types $P=\{(\sigma_n^{(2)},\sigma_n^{(2')})\}$, such that the union over $P$ of products of quadratic fields $K_1\times K_2$, where $K_1$ satisfies the splitting conditions of $\sigma_n^{(2)}$ and $K_2$ satisfies the splitting conditions of $\sigma_n^{(2')}$, is exactly $F$.  Let $\phi:\R_{\geq 0}\to\R_{\geq 0}$ be a compactly supported function that is $1$ is the interval $[0,1]$. An application of the hyperbola method allows us to compute the sum over $K_1\times K_2$ corresponding to $(\sigma_n^{(2)},\sigma_n^{(2')})$ to be
\begin{equation}\label{eq:22count_DH}
\begin{array}{rcl}
&&\displaystyle\sum_{K_1}\sum_{K_2}
\psi\Bigl(\frac{|\Delta(K_2)|}{X/|\Delta(K_1)|}\Bigr)\phi\Bigl(\frac{|\Delta(K_1)|}{\sqrt{X}}\Bigr)+
\sum_{K_2}\sum_{K_1}
\psi\Bigl(\frac{|\Delta(K_1)|}{X/|\Delta(K_2)|}\Bigr)
\phi\Bigl(\frac{|\Delta(K_2)|}{\sqrt{X}}\Bigr)
\\[.2in]
&&\displaystyle -
\sum_{(K_1,K_2)}
\psi\Bigl(\frac{|\Delta(K)||\Delta(K')|}{X}\Bigr)\theta\Bigl(\frac{|\Delta(K)|}{\sqrt{X}},\frac{|\Delta(K')|}{\sqrt{X}}\Bigr),
\end{array}
\end{equation}
for some smooth compactly supported function $\theta:\R^2_{\geq 0}\to\R$.
We begin with an estimate for the first summand above.
Applying Lemma \ref{lem:quad_field_count}, yields
\begin{equation*}
\sum_{K_2}
\psi\Bigl(\frac{|\Delta(K_2)|}{X/|\Delta(K_1)|}\Bigr)=
c_2(\sigma_n^{(2')})\frac{X}{|\Delta(K_1)|}+O\bigl(X^{1/2+o(1)}|\Delta(K_1)|^{-1/2}n^{1/4}\bigr),
\end{equation*}
for some constant $c_2(\sigma_n^{(2')})$ independent of $K_1$. When the error is summed over $K_1$ in the outer sum, we get a total error of $O(X^{3/4+o(1)}n^{1/4})$, which is sufficiently small. Meanwhile, summing the main term over $K_1$ in the outer sum yields 
\begin{equation*}
c_2(\sigma_n^{(2')}) X\cdot \sum_{K_1\in F_2(\sigma_n^{(2)})}\frac{1}{|\Delta(K_1)|}\phi\Bigl(\frac{|\Delta(K_2)|}{\sqrt{X}}\Bigr).
\end{equation*}
We evaluate the sum above using a sieve identical to that used in \eqref{eq:inc_exc_quad}. Following the notation there, we write
\begin{equation}\label{eq:22-count-temp-1}
\begin{array}{rcl}
\displaystyle\sum_{K_1\in F_2(\sigma_n^{(2)})}\frac{1}{|\Delta(K_1)|}\phi\Bigl(\frac{|\Delta(K_2)|}{\sqrt{X}}\Bigr)
&=&\displaystyle\sum_{(q,n)=1}\mu(q)\sum_{\substack{D\in S_q\\\pm D>0}}
\frac{1}{|D|}\phi\Bigl(\frac{|D|}{\sqrt{X}}\Bigr)
\\[.2in]&=&\displaystyle
\sum_{(q,n)=1}\mu(q)\sum_{D'\in\Z_{>0}}\frac{\theta_n(D')}{q^2n_0D'}
\phi\Bigl(\frac{D'}{X/(n_0q^2)}\Bigr)
\\[.2in]&=&\displaystyle
\frac{1}{n_0}
\sum_{(q,n)=1}\frac{\mu(q)}{q^2}\sum_{D'\in\Z_{>0}}\frac{\theta_n(D')}{D'}
\phi\Bigl(\frac{D'}{X/(n_0q^2)}\Bigr).
\end{array}
\end{equation}
As noted previously, the reduction of $\theta_n$ modulo any prime dividing $n$ either picks out the quadratic residues or the quadratic non residues. Using the (by now) standard Mellin inversion technique, we can express the inner sum over $D'$ in the last line above by a size $O(n^\epsilon)$ linear combination (with bounded coefficients) of integrals involving $L$-functions $L(s,\chi)$, where $\chi$ is a quadratic character modulo $n$. Each integral is of the form
\begin{equation*}
\int_{Re(s)=2} L(s+1,\chi)\Bigl(\frac{X}{n_0q^2}\Bigr)^s\wt{\phi}(s)ds.
\end{equation*}
We shift the integral left, picking up the pole at $s=0$, and moving past to $s=-1/2+\delta$ for some small $\delta$. When $\chi$ is a non-trivial character, the pole at $s=0$ is simple, coming from $\wt{\phi}(s)$. We merely pick up a constant, and a special $L$-value at $s=1$ which is bounded by $O(n^{o(1)})$. However, when $\chi$ is trivial, we have a double pole at $s=0$, giving us a $\log X$ contribution as well as a constant contribution. Hence the inner sum over $D'$ in the third line of \eqref{eq:22-count-temp-1} is equal to
\begin{equation*}
c_q'\log X+c_q+O(X^{-1/2+\delta}q^{1-2\delta}n^{1/2}).
\end{equation*}
The sum of the error term over $q$ converges. Since $\delta$ was arbitrary, we see that the first summand (and hence by symmetry, the second summand as well) of \eqref{eq:22count_DH} is of the form
$c'X\log X+cX$, up to sufficiently small error, where $c\ll n^{o(1)}$.

It is only left to evaluate the third summand of \eqref{eq:22count_DH}. We do this using a squarefree sieve, indeed a $2$-dimensional version of the sieve used in the proof of Lemma \ref{lem:quad_field_count}: adapting the notation there for our situation, we write
\begin{equation*}
\begin{array}{rcl}
&&\displaystyle\sum_{(K_1,K_2)}
\psi\Bigl(\frac{|\Delta(K)||\Delta(K')|}{X}\Bigr)\theta\Bigl(\frac{|\Delta(K)|}{\sqrt{X}},\frac{|\Delta(K')|}{\sqrt{X}}\Bigr)
\\[.2in]&=&
\displaystyle\sum_{\substack{q_1,q_2\\(q_1q_2,n)=1}}\mu(q_1q_2)
\sum_{\substack{D_1\in S_{q_1}^{(1)}\\\pm D_1>0}}
\sum_{\substack{D_2\in S_{q_2}^{(2)}\\\pm D_2>0}}
\psi\Bigl(\frac{|\Delta(K)||\Delta(K')|}{X}\Bigr)\theta\Bigl(\frac{|\Delta(K)|}{\sqrt{X}},\frac{|\Delta(K')|}{\sqrt{X}}\Bigr).
\end{array}
\end{equation*}
We break up the set of pairs $(q_1,q_2)$ into double dyadic ranges indexed by pairs $(Q_1,Q_2)$, where in each double dyadic range, $q_1$ and $q_2$ range from $[Q_1,2Q_1]$ and $Q_2,2Q_2]$, respectively. The two methods (the uniformity estimate and twisted Poisson summation) used in \eqref{eq:inc_exc_quad} bound the error contribution from each double dyadic range by
\begin{equation*}
X^{o(1)}\min\Bigl(\frac{\sqrt{X}}{Q_1},\sqrt{n}Q_1\Bigr)\min\Bigl(\frac{\sqrt{X}}{Q_2},\sqrt{n}Q_2\Bigr).
\end{equation*}
Optimizing, we obtain a total error of $O(X^{1/2+o(1)}n^{1/2})$, which is sufficiently small. Meanwhile, the main term is some absolutely bounded constant times $X$. The result therefore follows.
\end{proof}

\subsection{From counting fields to counting orders}

We define {\it congruence $(112)$-}, {\it congruence $(13)$-}, {\it congruence $D_4$-}, and {\it congruence $(22)$-} families of quartic \'etale algebras to be the set of all quartic \'etale algebras over $\Q$, satisfying finitely many prescribed local conditions, such that every member of the family is of the type $\Q\times\Q\times K_2$, $\Q\times K_3$, a $D_4$-field, and $K\times K'$, respectively, where $K_2$, $K$, and $K'$ are quadratic fields, and $K_3$ is a cubic field. Let $F$ be a congruence $T$-family, for $T\in\{(112),(13),D_4, (22)\}$, satisfying local conditions at the places in $S$, a finite set of places over $\Q$.
That is, for each $v\in S$, we let $\Sigma_v$ be a (finite) set of $T$-compatible algebras over $\Q_v$, and let $F=F_T$ be the set of all $T$-type $\Q$-algebras such that $F\otimes\Q_v\in\Sigma_v$ for each $v\in S$.

For each finite prime $p\in S$, let $\Lambda_p$ be any set of pairs $(Q_p,C_p)$, where $Q_p$ is an order in $K_p$ and $C_p$ is a cubic resolvent of $Q_p$, and denote the collection $(\Lambda_v)_{v\in S}$ by $\Lambda$ (where $\Lambda_v=\{K_v\}$ for the infinite place $v$).
Let $R_F(\Lambda)$ denote the set of all pairs $(Q,C)$, where $Q$ is a quartic order of some algebra in $F$, and $C$ is a cubic resolvent of $Q$, such that $(Q\otimes\Q_v,C\otimes\Q_v)\in\Lambda_v$ for all $v\in S$. For a positive squarefree integer $q$ (fixed), let $R_{F,q}(\Lambda)$ denote the set of elements in $R_F(\Lambda)$ that are non-maximal at every prime dividing $q$. 

The goal of this subsection is to carry out a smoothed count of the elements in $R_{F,q}(\Lambda)$, from which the main result will be quickly deduced. 
Let $\psi:\R_{\geq 0}\to\R_{\geq 0}$ be a smooth function with compact support. We begin by writing
\begin{equation*}
N^{(i)}_\psi(\phi_{q,\Lambda};X)=\sum_T\sum_{(Q,C)\in R_{F_T,q}(\Lambda)}\frac{1}{\Aut(Q)}\psi\Bigl(\frac{|\Delta(Q)|}{X}\Bigr) = \sum_T\sum_{q\mid n}\sum_{K\in F_T}c_K(n,\Lambda)\psi\Bigl(\frac{n^2|\Delta(K)|}{X}\Bigr),
\end{equation*}
where $c_K(n,\Lambda)$ denotes the weighted number of pairs $(Q,C)$, satisfying the constraints of $\Lambda$, where $Q$ is a suborder of $\O_K$ of index $n$ and $C$ is a cubic resolvent of $Q$, and $(Q,C)$ is weighted by $1/\Aut(Q)$. When $\Lambda$ consists of the empty set of local conditions, we denote $c_K(n,\Lambda)$ by $c_K(n)$. We have the following result on the numbers $c_K(n,\Lambda)$ and $c_K(n)$.
\begin{proposition}\label{prop:Nak}
Let $K$ be an \'etale quartic algebra over $\Q$, with ring of integers $\O_K$ and discriminant $D$. Then $c_K(n,\Lambda)$ only depends on $n$, $\Lambda$, and the splitting type of $K$ at every prime dividing $n$. Moreover, have $c_K(n,\Lambda)\leq c_K(n)\ll_\epsilon n^\epsilon N(n,D)$, where
\begin{equation}\label{eq:NakBound}
N(n,D):=
\prod_{\substack{p^2\nmid D\\p^e||n,\ e\geq 3}}p^{\lfloor e/2\rfloor}
\prod_{\substack{p^2\mid D\\p^e||n,\ e\geq 2}}p^{\lfloor e/2\rfloor}.
\end{equation}
\end{proposition}
\noindent Note that we have $N(n,D)=N(n,(n,D))$. The first claim of the result is clear; indeed, the function $c_K(n,\Lambda)$ is multiplicative, and when $n$ is a power of $p$ depends only on the power of $p$ dividing $n$, $\Lambda_p$ (when $p\in S$), and $K\otimes\Q_p$. The second claim is due to works of Nakagawa \cite{nakagawa} and Bhargava \cite{BHCL3}, and is proved in \cite[Corollary 11.3]{ST_second_main_term_1}. We are now ready to prove the main result of this section.

\medskip

\noindent{\bf Proof of Theorem \ref{thm_red_points_count}:}
We fix $T\in\{(112),(13),D_4,(22)\}$ and denote $T$-compatible quartic splitting types at all primes dividing a positive integer $n$ by $\sigma_n=(\sigma_p)_{p\mid n}$. We denote the contribution to $N^{(i)}_\psi(\phi_{q,\Lambda};X)$ from suborders of $F_T$ by $N_T$. Since $c_K(n,\Lambda)$ depends only on $n$, $\Lambda$, and the splitting types of $K$ at primes dividing $n$, we may write $c_K(n,\Lambda)=c_\Lambda(n,\sigma_n)$ for every algebra $K\in F$ having splitting type $\sigma_n$.
Therefore, for a constant $\lambda>0$ to be optimized later, we have
\begin{equation}\label{eq:summing_suborders_temp}
N_T = \sum_T\sum_{\substack{n\leq X^{1/6+\lambda}\\q\mid n}}\sum_{\sigma_n}c_\Lambda(n,\sigma_n)\sum_{\substack{K\in F_T\\\sigma_n(K)=\sigma_n}}\psi\Bigl(\frac{|\Delta(K)|}{X/n^2}\Bigr)
+O(X^{5/6-\lambda+o(1)}),
\end{equation}
where the sum over $n>X^{1/6+\lambda}$ is bounded via following lemma:
\begin{lemma}
The number of quartic triples $(Q,C,r)$ with discriminant less than $X$ and index greater than $M$ is bounded by $O(X^{1+o(1)}/M)$.    
\end{lemma}
\begin{proof}
The proof uses two inputs: first, the bound of $O(Y^{1+o(1)})$ on the number of nondegenerate \'etale quartic extensions of $Q$ with discriminant bounded by $Y$. Second, given an \'etale quartic algebra $K$ over $\Q$, the bound of $O(Z^{1+o(1)})$ on the number of quartic triples $(Q,C,r)$, where $Q$ is a suborder of the ring of integers of $K$ having index bounded by $Z$. With these two inputs at hand, we proceed as follows. If $(Q,C,r)$
\end{proof}

In the previous subsection, we prove for every choice of $T$, that the smooth sum over $K\in F$ satisfying specified splitting conditions at primes dividing $n$ has the following expansion:
\begin{equation}\label{eq:sum_algebras_expansion}
\sum_{\substack{K\in F_T\\\sigma_n(K)=\sigma_n}}\psi\Bigl(\frac{|\Delta(K)|}{Y}\Bigr)=\sum_{r\in\{1,5/6\}}\bigl(C_r'(\sigma_n,F)\log Y+C_r(\sigma_n,F)\bigr)Y^r+O(Y^{\delta+o(1)}n^\theta),
\end{equation}
for some constants $C_r'(\sigma_n,F)$, $C_r(\sigma_n,F)$, $\delta$, and $\theta$. 
Using \eqref{eq:sum_algebras_expansion} to evaluate the innermost sum in \eqref{eq:summing_suborders_temp}, we obtain
\begin{equation*}
\begin{array}{rcl}
\displaystyle N_T &=&\displaystyle
\sum_{r\in\{1,5/6\}}\bigl(C_{r,q}^{\red '}(T)\log X+C_{r,q}^\red(T)\bigr)X^r
+O_q(E_T),
\end{array}
\end{equation*}
where the constants $C_{r,q}^{\red '}(T)$ and $C_{r,q}^\red(T)$ are given by
\begin{equation}\label{eq:T_cont_rings}
\begin{array}{rcl}
C_{r,q}^{\red '}(T) &=& \displaystyle\sum_{\substack{n\geq 1\\q\mid n}}\sum_{\sigma_n}c(n,\sigma_n)\frac{C'_r(\sigma_n,F)}{n^{2r}},
\\[.2in]
C_{r,q}^\red(T) &=& \displaystyle
\displaystyle\sum_{\substack{n\geq 1\\q\mid n}}\sum_{\sigma_n}c(n,\sigma_n)\frac{C_r(\sigma_n,F)-2C_r'(\sigma_n,F)\log n}{n^{2r}},
\end{array}
\end{equation}
and the error term $E$ is bounded by
\begin{equation*}
E_T\ll_q X^{5/6-\lambda+o(1)}+X^{\delta+o(1)}\sum_{n\leq X^{1/6+\lambda}}
n^{\theta-2\delta}\ll X^{5/6-\lambda+o(1)}+X^{\delta}+X^{1/6+2\delta/3+\theta/6+\lambda(1+\theta-2\delta)}.
\end{equation*}
Note that the sums defining the constants $C_{r,q}^{\red '}(T)$ and $C_{r,q}^\red(T)$ converge absolutely thanks to the previously obtained bounds on $C_r'(\sigma_n,F)$, and the bounds on $c(n,\sigma_n)$ implied by Proposition~\ref{prop:Nak}. For $T\in\{(112),(13),D_4,(22)\}$, we see from Lemma \ref{lem:quad_field_count}, Theorem \ref{thm:cubic_field_count}, Theorem \ref{thm:D4-fields}, Proposition~\ref{prop:22_count}, respectively, we see that the values of the arising $(\delta,\theta)$ are $(1/2,1/4)$, $(2/3,2/3)$, $(3/5,2/5)$, $(1/2,1/2)$ and $(3/4,1/4)$, respectively, and that $D_{5/6}'$ is always $0$. 
We optimize the value of $\lambda$ in each of these cases, (taking $\lambda=1/12$ for each case actually suffices), obtaining $E_T\ll_q X^{3/4+o(1)}$. We have thus proven the first claim of Theorem \ref{thm_red_points_count}.
Finally, we use \eqref{eq:T_cont_rings} to deduce that
\begin{equation*}
C_{r,q}^{\red '}(T)\ll q^{-2};\quad C_{r,q}^\red(T)\ll q^{-5/3}.
\end{equation*}
Summing over types $T$ immediately yields the second claim of Theorem \ref{thm_red_points_count}. $\Box$

\medskip

We end with the following result.
Let $\Sigma=(\Sigma_v)_{v\in S}$ be a finite collection of local specifications for \'etale quartic algebras over $\Q$. Let $F^\red(\Sigma)$ denote the family of \'etale quartic extensions $K$ of $\Q$ such that for every $v\in S$ we have $K\otimes\Q_v\in\Sigma_v$, and such that $K$ is not an $S_4$- or $A_4$-extension. For a smooth function $\psi:\R_{\geq 0}\to\R$, we define
\begin{equation*}
N^\red_\Sigma(\psi,X):=
\sum_{K\in F^\red(\Sigma)}\psi\Bigl(\frac{|\Delta(K)|}{X}\Bigr).
\end{equation*}
We know that $N^\red_\Sigma(\psi,X)$ has a power series expansion
\begin{equation*}
N^\red_\Sigma(\psi,X)=\sum_{r\in\{1,5/6\}}\bigl( c'_r(\Sigma)\log X+c_r(\Sigma)\bigr)X^r+O(X^{3/4+o(1)}),
\end{equation*}
for some constants $c_r(\Sigma)$ and $c_r(\Sigma)$.
Let $\Lambda=(\Lambda_v)_{v\in S}$ be the finite collection of local specifications on quartic rings associated to $\Sigma$, i.e., for the primes $p\in S$, we set $\Lambda_p$ to be the set of maximal orders of $\Sigma_p$ and for the infinite place $v$, we set $\Lambda_v=\Sigma_v$. Then we have the following consequence of the proof of Theorem \ref{thm_red_points_count}.

\begin{corollary}\label{cor:fields_from_rings_from_fields}
With notation as above, we have $c_{5/6}'(\Sigma)=0$ and 
\begin{equation*}
c'_1(\Sigma)=\sum_q\mu(q)C_{1,q}^{\red '}(\Lambda,\psi);\quad
c_r(\Sigma)=\sum_q\mu(q)C_{1,q}^{\red '}(\Lambda,\psi).
\end{equation*}
\end{corollary}
\begin{proof}
We have determined power saving expansions, up to an error of $O(X^{3/4+o(1)})$, for the smoothed number of \'etale quartic extensions of $\Q$ of types $D_4$, $(22)$, $(13)$, and $(112)$. The remaining types are negligible in number. These power saving expansions have no $X^{5/6}\log X$ terms in them, and so the first claim follows immediately. For the second claim, we proceed type by type. The asymptotics of the fields count in $F^\red(\Sigma)$ having type $T$ can be read off from \eqref{eq:sum_algebras_expansion}, by setting $n=1$ (in which case $\sigma_n(K)=\sigma_n$ becomes an empty condition). Meanwhile, the contribution $C_{1,q}^{\red '}(T)$ and $C_{r,q}^{\red}(T)$ from a type $T$ to $C_{1,q}^{\red '}$ and $C_{r,q}^{\red}$, respectively, are given in \eqref{eq:T_cont_rings}. We write
\begin{equation*}
\begin{array}{rcl}
\displaystyle\sum_q \mu(q)C_{1,q}^{\red '}(T)
&=&
\displaystyle\sum_q\mu(q)\sum_{\substack{n\geq 1\\q\mid n}}\sum_{\sigma_n}c(n,\sigma_n)\frac{C'_1(\sigma_n,F)}{n^{2r}}
\\[.2in]&=&
\displaystyle\sum_n\sum_{\sigma_n}c(n,\sigma_n)\frac{C'_1(\sigma_n,F)}{n^{2r}}\sum_{\substack{q\mid n}}\mu(q)
\\[.15in]
&=&\displaystyle c(1,\sigma_1)C'_1(\sigma_1,F),
\end{array}
\end{equation*}
and similarly for the sums over $q$ of $C_{r,q}^\red(T)$. The corollary now follows since $c(1,\sigma_1)$, the number of suborders of a maximal order of index-$1$, is equal to one, and this suborder (being maximal) has a unique cubic resolvent.
\end{proof}

\section{Bounding reducible elements in the generic body}
In this section, we estimate the number of reducible elements in the  generic body. Specifically, we prove the following result.
Let $\chi_{\ngen,\gb}$ denote the characteristic function of the set of non-generic elements $V(\Z)$ that lie in the  generic body (i.e., outside the slices $a_{11}=b_{11}=0$ and $\det(A)=0$). Given a periodic $G(\Z)$-invariant function $\phi:V(\Z)\to\R$, let $\phi^{\ngen,\gb}$ denote $\phi\cdot \chi_{\ngen,\gb}$. Then we have the following result.

\begin{theorem}\label{th_red_main_body_final}
With notation as above, we have
\begin{equation*}
\I_\eta(\phi^{\ngen,\gb},\cB;X)=C^{\ngen,gb}_\phi X^{5/6}+O(X^{5/6-\delta}),
\end{equation*}
for some $\delta>0$. Moreover, if $q$ is a squarefree integer relatively prime to $N$, the modulus defining $\phi$, and if $\phi_q$ denotes the product of $\phi$ with the characteristic function of the set of elements in $V(\Z)$ which are nonmaximal at $q$, then we have 
\begin{equation*}
C^{\ngen,gb}_{\phi_q}=C^{\ngen,gb}_{\phi}\Bigl(\frac{1}{q}+O(q^{-2})\Bigr).
\end{equation*}
\end{theorem}

In \S4.1, we obtain sufficiently small bounds on the number of elements with reducible resolvent, and in \S4.2, we obtain asymptotics on the number of reducible elements. Combining them immediately yields Theorem \ref{th_red_main_body_final}.

\subsection{Elements with reducible cubic resolvent}

\begin{proposition}\label{prop_reducible_bound_resolvent}
Let $\phi$ denote the characteristic function of the set of elements $(A,B)\in V(\Z)$ such that $\Res(A,B)$ has a rational root and $\det(A)\neq 0$. Let $\cB:V(\R)\to\R$ be a smooth function with compact support away from the discriminant-0 locus. Then we have
\begin{equation}\label{eq:red_res_prop}
\I_\eta(\phi,\cB;X)=\displaystyle\int_{g\in\FF}\sum_{x\in V(\Z)}\phi(x)(g\cB)(x)\eta\Big(\frac{\lambda(g)}{X^{1/12}}\Big)dg = O(X^{23/28+o(1)}).
\end{equation}
\end{proposition}
\noindent This result is an improved version of \cite[Lemma 12]{dodqf}, which obtained a bound of $O(X^{11/12})$.

\medskip

\begin{proof}
We prove the proposition in multiple steps. Let $S$ denote the set of elements $x\in V(\Z)$ that are counted with non-zero weights in the left hand side of the equation of the proposition. First, we claim that the number of integral binary cubic forms that occur as resolvents of $S$ is bounded by $O(X^{3/4+o(1)})$. Indeed, any resolvent binary cubic form with coefficients $a$, $b$, $c$, and $d$ which occurs must be reducible, must satisfy, for some $1\ll t\ll X^{1/12}$, the inequalities 
\begin{equation*}
a\ll t^{-3}X^{1/4};\quad b\ll t^{-1}X^{1/4};\quad c\ll tX^{1/4};\quad d\ll t^3X^{1/4},
\end{equation*}
and must satisfy $a\neq 0$.
The number of choices for the triple $(a,b,d)$ is hence bounded by $O(X^{3/4+o(1)})$. Once chosen, there are at most $O(X^{o(1)})$ choices for the roots of the reducible cubic (since they must divide $ad$). Once $(a,b,d)$ and the roots are fixed, this determines $c$ proving the claim.
It follows that the number of discriminants $\Delta(x),x\in S$ is bounded by $O(X^{3/4+o(1)})$.

Given any element $x\in S$, we associate to it the triple $(K_4,K_3,D)$, where $K_4$ is the quartic algebra corresponding to $x$, $K_3$ is the (reducible) cubic $\Q$-algebra corresponding to $\Res(x)$, and $D=\Delta(x)=\Delta(\Res(x))$. Let $S_1$ denote the set of triples we thus obtain. In our second step, we prove two facts about $S_1$: first that the map $S_1\to\Z$ sending $(K_4,K_3,D)$ to $D$ has fibers of size $O(X^{o(1)})$, and second, that $|S_1|\ll X^{3/4+o(1)}$. To obtain the first claim, we note that $D$ determines the reducible algebra $K_3=\Q\times\Q(\sqrt{D})$ uniquely. Next note that if $K_3$ is a reducible cubic algebra, corresponding to the resolvent of an element $x\in V(\Z)$, then the quartic $\Q$-algebra corresponding to $x$ has at most $O(|\Delta(x)|^{o(1)})$ choices. This is proven in \cite[Lemma 12]{dodqf} (and hinges on the fact that the $2$-torsion of the class groups of quadratic fields is small). This yields the first claim, and the second claim now follows immediately from Step 1, which implies that the number of discriminants $D$ that can arise is bounded by $O(X^{3/4+o(1)})$. 

Third, we fix an element $D\in \Z$, and estimate the weighted sum (in the middle term of the equation of the proposition) of elements $x\in S$ that have discriminant $D$. To do this, we first observe this value of $D$ fixes $K_4$ up to $O(X^{o(1)})$ possibilities. Fixing $D$ and $K_4$, we note that Proposition~\ref{prop:I_to_N} implies that this weighted sum is bounded by the number of $G(\Z)$-orbits on reducible elements with discriminant $D$ and associated quartic algebra $K_4$. It thus suffices to bound the number of pairs $(Q,C)$, where $Q$ is a suborder of $K_4$, $C$ is a cubic resolvent ring of $Q$, and $\Delta(Q)=D$. Write $D=dq^2$, where $q$ is cubefull and as large as possible (this determines $q$ and $d$ uniquely). It follows from Proposition \ref{prop:Nak} that the number of sub-pairs $(Q,C)$ of $(K_4,K_3)$ of discriminant $D$ is bounded by ${\rm sqr} (q)|D|^{o(1)}$ for the multiplicative function ${\rm sqr}$ defined on prime powers by ${\rm sqr}(p^e):=p^{\lfloor e/2\rfloor}$.

We write $q=q_3^3q_4$, where $q_3$ is square-free and $q_4$ has only prime power of exponent at least 4 - i.e. $q_4$ is 4-full. We break up the ranges of $q_3$ and $q_4$ into dyadic regions of size $Q_3$ and $Q_4$, respectively. By \cite{Ivic78} the number of $q_4$ in this dyadic range is asymptotic to  $Q_4^{1/4}$. Hence,  for each such region, the number of $D$ is bounded by 
$$\min\Big(X^{3/4+o(1)},\frac{X^{1+o(1)}}{Q_3^5Q_4^{7/4}}\Big).$$ The number of reducible pairs with discriminant $D$ is bounded by ${\rm sqr}(q)\leq q_3q_4^{1/2}$. Hence, the contribution from discriminants $D$ in this dyadic range are bounded by
$$\min\Big(X^{3/4+o(1)}Q_3Q_4^{1/2},\frac{X^{1+o(1)}}{Q_3^4Q_4^{5/4}}\Big).$$ Optimizing at $Q_4=X^{\frac1{32}}, Q_3=1$, and summing over dyadic ranges gives the bound of 
 $O(X^{23/28+o(1)})$ as needed.
\end{proof}

We have the following immediate consequence of the proof of the above proposition.
\begin{corollary}\label{cor_reducible_bound_resolvent}
Let $\psi$ and $\cB$ be as in Proposition \ref{prop_reducible_bound_resolvent}. Let $\phi^{\rm rr}$ denote the characteristic function of the set of elements in $V(\Z)$ with reducible resolvent. Then we have
\begin{equation*}
\begin{array}{rcl}
\displaystyle\int_{\substack{g\in\FF\\t\ll X^\delta}}\sum_{x\in V(\Z)}\phi^{\rm rr}(g\cB)(x)\psi\Big(\frac{\lambda(g)}{X^{1/12}}\Big)dg
&=&O(X^{23/28+3\delta+o(1)}).
\end{array}
\end{equation*}
\end{corollary}
\begin{proof}
As before, let $S$ denote the set of elements $x\in V(\Z)$ that are counted with non-zero weights in the left hand side of the equation of the proposition. This time, we claim that the number of integral binary cubic forms that occur as resolvents of $S$ is bounded by $O(X^{3/4+3\delta+o(1)})$. Indeed, any resolvent binary cubic form with coefficients $a$, $b$, $c$, and $d$ which occurs must be reducible, must satisfy, for some $1\ll t\ll X^{\delta}$, the inequalities 
\begin{equation*}
a\ll t^{-3}X^{1/4};\quad b\ll t^{-1}X^{1/4};\quad c\ll tX^{1/4};\quad d\ll t^3X^{1/4}.
\end{equation*}
If $a\neq 0$, then the number of choices for $(a,b,d)$ is again $O(X^{3/4+o(1)})$, and once fixed, this triple determines $c$ up to $O(X^{o(1)})$ choices. If $a=0$, then there are $O(X^{3/4+3\delta})$ choices for the triple $(b,c,d)$ and the claim follows. The rest of the proof follows identically to the proof of Proposition \ref{prop_reducible_bound_resolvent}.
\end{proof}

\subsection{Elements with a common rational zero in $\P^2(\Q)$}

We prove the following result regarding elements reducible elements $(A,B)\in V(\Z)$ (i.e., elements such that $A$ and $B$ have a common rational zero in $\P^1(\Q)$) in the generic body.
\begin{proposition}\label{prop_reducible_bound_P2Q}
Let $\phi=\phi_\red\cdot\phi_\fin$, where $\phi_\red$ is the characteristic function of the set of elements $(A,B)\in V(\Z)$ such that $A$ and $B$ have a common rational zero, $\det(A)\neq 0$, and $(a_{11},b_{11})\neq (0,0)$, and $\phi_\fin$ is a $G(\Z)$-invariant function on $V(\Z)$ defined by finitely many congruence conditions. Let $\cB:V(\R)\to\R$ be a function with compact support away from the discriminant-0 locus. Then we have
\begin{equation}\label{eq_prop_reducible_2}
\cI_\eta(\phi,\cB;X)=\int_{g\in\FF}\sum_{x\in V(\Z)}\phi(x)(g\cB)(x)\eta\Big(\frac{\lambda(g)}{X^{1/12}}\Big)dg = C_\red(\phi) X^{5/6}+O(X^{5/6-\theta}),
\end{equation}
for a positive constant $\theta$, and $C_\red(\phi)=C_\red\cdot C(\phi_\fin)$, where $C_\red=C_\red(\eta,\cB)$ is a constant independent of $\phi_{\fin} $and $C(\phi_\fin)$ is the density of $\phi_\fin$ on the set $\{(A,B)\in V(\Z):a_{11}=b_{11}=0\}$.
\end{proposition}
\noindent This result is an improved version of \cite[Lemma 13]{dodqf}, which obtained a bound of $O_\epsilon(X^{11/12+\epsilon})$.

We begin with some preliminary lemmas.

\begin{lemma}\label{lem:succ_min}
Let $g\in \FF$ have Iwasawa torus component $(t,s_1,s_2)$. Then any element $x\in V(\Z)$ with $(g\cB)(x)\neq 0$ corresponds to a quartic ring, along with a basis $\langle 1,\alpha_1,\alpha_2,\alpha_3\rangle$, satisfying the following properties. First, this basis is an almost Minkowski basis for the quartic ring. Second, the lengths of the $\alpha_i$ are respectively within $($absolutely bounded$)$ constant factors of $1$, $s_1^{-2}s_2^{-1}\lambda^2$, $s_1s_2^{-1}\lambda^2$, and $s_1s_2^{2}\lambda^2$.
\end{lemma}
\begin{proof}
Any element $(A,B)\in V(\R)$ with nonzero discriminant gives rise to a quartic $\R$-algebra $Q_\R$ along with a basis $\mathbb B:=\langle\alpha_1,\alpha_2,\alpha_3\rangle$ of $Q_\R/\R$. The action of $\SL_2(\R)$ on $V(\R)$ leaves $\mathbb B$ unchanged, and the action of $\SL_3(\R)$ on $V(\R)$ respects the natural action of $\SL_3(\R)$ on the set of possible basis' $\mathbb B$. When $(A,B)$ is an element of $V(Z)$, these elements $\alpha_i$ are elements in $Q$, the quartic ring over $\Z$ corresponding to $(A,B)$, and give a basis for $Q/\Z$. Since $\cB$ has compact support away from the discriminant $0$ locus, the basis vectors corresponding to any $x$ in the support of $\cB$ have lengths which are $\asymp 1$, and pairwise angles $\asymp 1$ (by $\asymp 1$, we mean the quantity is bounded, and bounded away from $0$). Let $g\in\FF$ be any element and let $x$ be an element in the support of $\cB$. Since the basis $\mathbb B$ corresponding to $g\cdot x$ is given by the action of $g$ on the basis corresponding to $x$, it follows that the pairwise angles between the vectors of $\mathbb B$ are also bounded away from $0$, and that their lengths are as given in the statement of the lemma. 
\end{proof}

Before we state our next lemma, we need the following notation. Let $(A,B)\in V(\Z)$ be an element with nonzero discriminant, where $A$ and $B$ have a unique common zero in $\P^2(\Q)$. Then the quartic ring $Q$ corresponding to $(A,B)$ is an order in the $\Q$-algebra $\Q\times K$, where $K$ is a cubic field. The maximal order $\cO$ of $\Q\times K$ has two idempotents. We denote the intersection of the $\Z$-module generated by these idempotents with $Q$ by $I(A,B)$.
\begin{lemma}\label{lem_red_index}
Let $(A,B)\in V(\Z)$ be an element with nonzero discriminant, where $A$ and $B$ have a unique common zero in $\P^2(\Q)$ at the point $[r_1:r_2:r_3]$. Let $(Q,C)$ be the pair of rings corresponding to the $G(\Z)$-orbit of $(A,B)$, and let $\langle\alpha_1,\alpha_2,\alpha_3\rangle$ be the basis of $Q/\Z$ corresponding to $(A,B)$. Then the following are true:
\begin{itemize}
\item[{\rm (1)}] The rank-$1$ module $I(A,B)/\Z\subset Q/\Z$ is generated by $r_1\alpha_1+r_2\alpha_2+r_3\alpha_3$;
\item[{\rm (2)}] Let $w_2$, $w_3$ complete the basis of $\Z^3$ started by $w=[r_1:r_2:r_3]$. To avoid confusion, we let $F_A$ and $F_B$ denote the bilinear forms corresponding to the matrices $A$ and $B$, respectively. Then the index of $Q$ in its maximal order is divisible by
\begin{equation*}
    \det\begin{pmatrix}
F_A(w,w_2) & F_A(w,w_3)\\F_B(w,w_2) & F_B(w,w_3).
\end{pmatrix}
\end{equation*}
\end{itemize}
\end{lemma}
\begin{proof}
The group $\SL_3(\Z)$ acts transitively on $\P^1(\Q)$, and so we only need to check the claims on the element $[1:0:0]$. The first claim is a consequence of the multiplication tables for $Q$ given by Bhargava in \cite[\S3.2]{BHCL3}. The second claim follows from the discussion following \cite[Lemma 22]{BHCL3} (specifically, \cite[Equation (43)]{BHCL3}).
\end{proof}

We are now ready to prove Proposition \ref{prop_reducible_bound_P2Q}:

\medskip

\noindent {\bf Proof of Proposition \ref{prop_reducible_bound_P2Q}:} We begin with some general observations. We first note, as we did in the proof of the previous proposition, that when the sum over $x$ in the middle term of \eqref{eq_prop_reducible_2} is restricted to any $G(\Z)$-invariant set, the integral can be bounded (via Proposition \ref{prop:I_to_N}) by instead bounding the number of $G(\Z)$-orbits on this set (with discriminant less than $X$). Also note that for the sum over $g\cB(x)$ to be nonzero in the middle term of \eqref{eq_prop_reducible_2}, we must have $w(b_{11}),w(\det(A)),w(a_{13}),w(a_{22})\gg 1$. This implies the following estimates on $g=n(s_1,s_2,t)k\in\FF$:
\begin{equation}\label{eq_red_prop_cond_s12t}
t\ll\lambda;\quad s_1^4s_2^2\ll t\lambda;\quad ts_1\ll s_2\lambda ;\quad ts_2^2\ll s_1\lambda.
\end{equation}
Also, in this proof, we will not worry about being optimal about the error term - any power saving will do. To this end, we fix a small positive constant $\delta$. 

We begin by bounding the sum in the middle term of \eqref{eq_prop_reducible_2} over those $x$ whose corresponding quartic ring is an order in $\Q\times\Q\times K$, where $K$ is a quadratic $\Q$-algebra. Note that the successive minima of the maximal order in such a $\Q$-algebra are $1$, $1$, $1$, and $m_4$, for some number $m_4$. This means (by Lemma \ref{lem:succ_min}) that the quartic ring corresponding to $x$, for any $x\in V(\Z)$ with $(s_1,s_2)\cB(x)\neq 0$, has index at least $s_1^{-1}s_2^{-2}\lambda^4$ in its maximal order. From a (non-optimal) application of \eqref{eq_red_prop_cond_s12t}, we obtain
\begin{equation*}
s_1s_2^2\ll (s_1^2s_2^2)^{1/2}\cdot (s_2^2)^{1/2}
\ll (t\lambda s_1^{-2})^{1/2}\cdot (t^{-1}\lambda s_1)^{\frac12}\ll \lambda s_1^{-1}\ll \lambda
\end{equation*}
Hence any quartic ring corresponding to $x$ has index $\gg \lambda^3\asymp X^{1/4}$. The number of $G(\Z)$-orbits on such $x$, with $0<|\Delta(x)|\ll X$ is bounded (as a consequence of Proposition~\ref{prop:Nak}) by $O(X^{3/4+o(1)})$, which is sufficiently small.

Hence, for the purpose of proving the result, we may restrict the sum over those $x$ in $V(\Z)$ whose associated quartic ring is an order in $\Q\times K$, for some cubic field $K$. Such $x=(A,B)$ are such that $A$ and $B$ have a unique common zero $[r_1:r_2:r_3]\in\P^1(\Q)$, where $[r_1:r_2:r_3]\neq [1:0:0]$. We define the {\it height} of $[r_1:r_2:r_3]$ to be $\max\{|r_1|,|r_2|,|r_3|\}$. Let $\FF'\subset \FF$ be the set of elements $g=n(t,s_1,s_2)k$ with either $t\geq X^{7\delta}$, $s_1\geq X^\delta$, or $s_2\geq X^\delta$. We structure the rest of our argument as ``bounding results'' followed by ``counting results''. We show that the integral over $\FF$ of the sum over those $x\in G(\Z)$ whose common zero in $\P^1(\Q)$ has height greater than $X^{4\delta}$ is negligible (a power saving on $X^{5/6}$). We show that the integral over $\FF'$ of the entire sum is negligible. To evaluate the main term, it thus suffices to sum over $x\in G(\Z)$ with a common zero of small height, and integrate this sum over $\FF\backslash \FF'$. We do this in the final step, by showing that it is a convergent sum. Furthermore, we break up our bounding results into two cases: the case when $r_3\neq 0$ and the case when $r_3=0$ (in which case $r_2\neq 0$).

\medskip

\begin{lemma}
The contribution to the middle term of \eqref{eq_prop_reducible_2} from the region $\FF'\subset\FF$ is $O(x^{5/6-\delta})$. The contribution from elements $x$ such that the corresponding common zero in $\P^2(\Q)$ has height greater than $X^{4\delta}$ is $O(x^{5/6-\delta})$.
\end{lemma}

\begin{proof}
We prove the lemma by partioning the set of $x$ being counted into two sets. Let $[r_1:r_2:r_3]$ denote the common zero in $\P^2(\Q)$ corresponding to $x$. The two sets are the set of $x$ for which $r_3\neq 0$, and the set of $x$ for which $r_3=0$.

\medskip

\noindent{\bf Step 1: Bounds when $r_3\neq 0$.}
From Lemmas \ref{lem:succ_min} and \ref{lem_red_index} (Part (1)), it follows that the index of any $x\in G(\Z)$, with $(g\cB)(x)\neq 0$, and associated common zero $[r_1:r_2:r_3]$ is at least of size $$\max\{|r_1|s_1^{-2}s_2^{-1}\lambda^2, |r_2|s_1s_2^{-1}\lambda^2,|r_3|s_1s_2^2\lambda^2\}.$$

Hence, if we restrict $g\in\FF$ to the range when either $s_1\geq X^\delta$ or $s_2\geq X^\delta$, we see that the index of the $x$ being counted are $\gg X^{1/6+\delta}$, giving a bound of $O(X^{5/6-\delta+o(1)})$ by Proposition \ref{prop:Nak}.
Restricting therefore to the set of $g\in\FF$ with $s_1,s_2< X^\delta$, we now note that if the height of the common zero is greater than $X^{4\delta}$, then again the index of the $x$ being counted is $\gg X^{1/6+\delta}$. We may thus restrict the integral to $g\in\FF$ with $s_1,s_2<X^\delta$ and restrict the sum over $x\in V(\Z)$ to those whose associated common zero has height less than $X^{4\delta}$. 

Let us impose this restriction, and further confine ourselves to the region where $t\gg X^{7\delta}$. Summing over these $X^{12\delta}$ possible common zeroes, and noting that fixing such a zero with $r_3\neq 0$ determines the value of $a_{33}$ and $b_{33}$ (given the rest of the values of $A$ and $B$), yields a bound of
\begin{equation*}
X^{12\delta}\int_{t>X^{7\delta}}\lambda^{10}t^{-2}d^\times t\ll X^{5/6-2\delta},
\end{equation*}
which is also sufficiently small.

\medskip

\noindent{\bf Step 2: Bounds when $r_3= 0$.} In this case, we must have $r_2\neq 0$. We begin by noting that we can now assume $w(a_{12})\gg 1$, since if $a_{11}=a_{12}=0$, then $A$ can only have a zero at $[r_1:r_2:0]$ when $r_2$ is also $0$.
We fiber over $a_{11}$, $a_{22}$, $b_{11}$, and $b_{22}$. These values determine $r_1$ and $r_2$ up to $O(X^{o(1)})$ possibilities, and once $r_1$ and $r_2$ are also determined, this uniquely determines $a_{12}$ and $b_{12}$. We may separate the integral over $\FF$ into two regions: the region where $a_{11}$ is forced to be $0$, and the region where it is not. These integrals are respectively bounded by $X^{o(1)}$ times
\begin{equation*}
\int_{t,s_1,s_2\ll P(\lambda)} \lambda^9 t^{-1}d^\times ts^\times s_1 d^\times s_2\quad\mbox{and}
\quad \int_{t,s_1,s_2} \lambda^{10} t^{-2}s_1^{-4}s_2^{-2}d^\times ts^\times s_1 d^\times s_2,
\end{equation*}
where $P(\lambda)$ is some polynomial in $\lambda$ bounding $t$, $s_1$, and $s_2$ (the existence of such a bound follows from \eqref{eq_prop_reducible_2}). It is immediately clear that if the integral is restricted to $\FF'\subset\FF$, the corresponding bound is improved to $O(X^{5/6-2\delta})$, which is sufficiently small.

We may thus restrict our integral to $\FF\backslash \FF'$. Doing so, and assuming that the height of the common zero $[r_1:r_2:0]$ is greater than $X^{4\delta}$ once again yields a sum over those $x\in V(\Z)$ whose associated quartic ring must have index at least $X^{1/6+\delta}$, by applying Lemma \ref{lem:succ_min}. Another application of Proposition \ref{prop:Nak} gives us the required bound.
\end{proof}

\medskip

Finally, we sum over those  common roots of small height. Let $g\in\FF\backslash\FF'$, let $r=[r_1:r_2:r_3]$ be a point in $\P^2(\Q)$ with height $h<X^{4\delta}$, and let $S_r$ denote the set of elements in $V(\Z)$ with a common zero at $r$. The condition that $(A,B)\in V(\Z)$ must have a common zero at $r$ imposes a linear condition on the coefficients of $A$ and on the coefficients of $B$ (in fact the same condition). This linear condition has coefficients of size $\leq h^2<X^{4\delta}$, and at least one coefficient of this linear condition has size $h^2$.

Since $g\in\FF\backslash \FF'$ is not high up in the cusp, as long as $\delta$ is sufficiently small, summing $g\cB$ over $S_r$ and integrating over $g\in\FF\backslash\FF'$ grows like $c_rX^{5/6}$ with a power saving error term, where $c_r\ll 1/h^4$ (a factor of $1/h^2$ coming from both the $a_{ij}$'s and the $b_{ij}$'s). The sum of $c_r$ over $r$ clearly converges, yielding Proposition \ref{prop_reducible_bound_P2Q} when $\phi_\fin=1$. For general functions $\phi_\fin$, we finish by noting that the density of $\phi_\fin$ on each set $S_r$ is the same: indeed, the group $\GL_3(\Z)$ acts transitively on $\P^2(\Q)$, and so to compute this density, we may simply move $r$ to $[1:0:0]$ and use the $G(\Z)$-invariance of $\phi_\fin$ to conclude. $\Box$

\medskip

Theorem \ref{th_red_main_body_final} follows from Propositions \ref{prop_reducible_bound_resolvent} and \ref{prop_reducible_bound_P2Q}, in conjunction with the following lemma.

\begin{lemma}\label{lem_density_nonmax_00}
Let $q$ be a squarefree integer and let $\phi_q$ be the characteristic function of the set of elements $(A,B)\in V(\Z)$ which are nonmaximal at $q$. Then $C(\phi_q)=1/q+O(1/q^2)$.
\end{lemma}
\begin{proof}
The proof follows from Part 2 of Lemma \ref{lem_red_index} since the probability that a $2\times 2$ matrix has determinant $0$ mod $q$ is $1/q+O(1/q^2)$. The probability that $(A,B)$ is nonmaximal at $q$ in some other way is $O(1/q^2)$, as can be verified with a computation identical to that of \cite[\S4.2]{BHCL3}.
\end{proof}

\section{Counting elements points in the generic body}

Let $\phi:V(\Z)\to\R$ be a $G(\Z)$-invariant function defined modulo a finite number. We let $\phi^\gb$ denote the product of $\phi$ with the characteristic function of the  generic body $V(\Z)^\gb$. (We recall that $V(\Z)^\gb$ denotes the set of elements $(A,B)\in V(\Z)$ with $\det(A)\neq 0$ and $(a_{11},b_{11})\neq (0,0)$.) In this section, we prove the following result.

\begin{theorem}\label{thm_point_count_generic}
Let notation be as above, let $\eta:\R_{\geq 0}\to\R_{\geq 0}$ be a smooth function with compact support away from $0$, and let $\cB:V(\R)^{(i)}\to\R$ be a smooth function with compact support away from the discriminant $0$ locus. Then we have
\begin{equation*}
\begin{array}{rcl}
\cI_\eta(\phi^\gb,\cB;X) &=&\displaystyle
\Vol_{dg}(\FF)\wt{\eta}(12)\nu(\phi)\Vol(\cB)X+\bigl(C_{00}(\phi)+C\wt{\eta}(10)Z\bigl((\cB)_{\emptyset;\{A\}},\nu(\phi_A);4/3\bigr)\bigr) X^{5/6}
\\[.2in]&&
+cX^{21/24}+O(X^{5/6-\theta'}),
\end{array}
\end{equation*}
for constants $c$ and $\theta'>0$, and where $C_{00}(\phi)$ is as in the discussion surrounding \eqref{eq:extra561}, and where $C$ is as in \cite[Proposition 7.1]{ST_second_main_term_1}.

Moreover, we have
\begin{equation*}
C_{00}(\phi_q)=C_{00}(\phi)\Bigl(\frac1{q}+O(q^{-2})\Bigr),
\end{equation*}
where $\phi_q$ is as in the statement of Theorem \ref{th_red_main_body_final}.
\end{theorem}

\subsection{Setup}

The results in this section will be obtained with a (small) modification of the results and methods of \cite{ST_second_main_term_1}. As there, we will proceed by first dividing the integral defining $\cI_\eta(\phi^\gb,\cB;X)$ into two pieces, corresponding to $t$ small and $t$ large, respectively.
To this end, let $f_0:\R_{\geq 0}\to \R_{\geq 0}$ be a smooth and compactly supported function satisfying $f_0(x)=1$ for $x\in[0,2]$. Define $f$ by setting $f(x)=1-f_0(x)$. Throughout this section, we fix $\delta>0$, which will be assumed to be small. For a real number $X>0$, define the functions $f^X:\R_{\geq 0}\to\R$ and $f_0^X:\R_{\geq 0}\to \R$ by setting
\begin{equation*}
f_0^X(x):=f_0(x/X^\delta);\quad f^X(x):=f(x/X^\delta).
\end{equation*}
For $g_2\in\FF_2$ with Iwasawa decomposition $g_2=(n,t,k)$, we define $f_0^X(g_2):=f_0^X(t)$ and $f^X(g_2):=f^X(t)$, and for $g=(\lambda,g_2,g_3)\in\FF$, we define $f_0^X(g):=f_0^X(g_2)$ and $f^X(g):=f^X(g_2)$.
We break up $\I_\eta(\phi^\gb,\cB;X)$ as $\I_\eta(\phi^\gb,\cB;X)=\I^{(1)}(\phi^\gb,\cB;X)+\I^{(2)}(\phi^\gb,\cB;X)$, where
\begin{equation}\label{eq:I1-def}
    \I^{(1)}(\phi^\gb,\cB;X)=\int_{g\in\FF}\sum_{x\in V(\Z)} \phi^{\gb}(x)(g\cB)(x)\eta\Big(\frac{\lambda(g)}{X^{1/12}}\Big)f_0^X(g)dg;
\end{equation}
\begin{equation}\label{eq:I1-def}
    \I^{(2)}(\phi^\gb,\cB;X)=\int_{g\in\FF}\sum_{x\in V(\Z)} \phi^{\gb}(x)(g\cB)(x)\eta\Big(\frac{\lambda(g)}{X^{1/12}}\Big)f^X(g)dg.
\end{equation}

We import notation on multiple zeta functions and smooth functions from \cite[\S4.1,4.4]{ST_second_main_term_1} which will be necessary in the rest of this section. We say $f:\Z^n\to\C$ is a periodic function if it is defined by congruence conditions modulo some positive integer. We use $a_1,\dots,a_n$ to denote the coordinates on $\Z^n$. Let $t\in\{\pm 1\}^n$. Writing $\vec{s}$ for $(s_1,\ldots,s_n)$, we define the multiple zeta function $\zeta_{f,t}(s_1,\ldots,s_n)$ associated to $f$ and $t$ by
$$\zeta_{f,t}(\vec{s}):= \sum_{t\cdot \vec{a}\in\Z_{>0}^n} f(\vec{a})\prod_{i=1}^n  |a_i|^{-s_i}.$$ 
Note that since $f$ is periodic, its values are absolutely bounded. Hence $\zeta_{f,t}(\vec{s})$ converges absolutely for $(s_1,\ldots,s_n)\in\C^n$ with $\Re(s_i)>1$ for each $i$.

\begin{definition}
This follows \cite[\S4.1]{ST_second_main_term_1}. Let $S,T\subset\{1,2,\dots,n\}$ be disjoint subsets. Denote the complement of $S\cup T$ by $R$.  For each element $v\subset \Z^T$ we define $f_{S;T}(v)$ to be the average value of  
$f$ on the set $\{0_S\}\times\{v\}\times\Z^R\subset\Z^n$. Here, by $\{0_S\}\times\{v\}\times\Z^R$, we mean the subset of elements $w\in\Z^n$ such that $a_i(w)=0$ for $i\in S$ and $a_i(w)=a_i(v)$ for $i\in T$. For $t\in\{\pm1\}^T$ and $s_T\in\C^T$, we then define
$$\zeta_{f,t}(S=0;T)(s_T):= \sum_{t\cdot \vec{a}\in \Z_{>0}^T} f_{S,T}(\vec{a})\prod_{i\in T} |a_i|^{-s_i}.$$ 
By convention, we will write $\zeta_{f,t}(T)(s_T)$ for $\zeta_{f,t}(\emptyset=0;T)(s_T)$ when $S$ is empty. When $T=S^c$, we write $\zeta_{f,t}(S=0)(s_T)$ for $\zeta_{f,t}(S=0;S^c)(s_T)$. Note that $\zeta_{f,t}(S=0;\emptyset)$ is simply a complex number, namely, the density of $f$ on the set $\{0_S\}\times\Z^{S^c}$. We will denote this density by $\nu(f|_S)$. If $L$ is a set whose characteristic function $\chi_L$ is periodic we shall write $\zeta_L$ for $\zeta_{\chi_L}$ and $\nu(L|_S)$ for $\nu(\chi_L|_S)$. We define $\zeta_f(S=0;T)$ to be the vector indexed by $(\pm 1)^{T}.$
\end{definition}

When we have a function $B$ in $n$ variables $x_1,\dots,x_n$, for each $\vec{t}\in(\pm)^n$ we set 
$$\tilde{B}_{\vec{t}}(\vec{s}) = \int_{\R_+} B(t_1x_1,\dots,t_nx_n)\prod_i x_i^{s_i} d^{\times}\vec{x}$$ and $\tilde{B}(\vec{s})$ to be the element in $\C^{(\pm)^n}$ whose co-ordinates are 
$\tilde{B}_{\vec{t}}(\vec{s})$.

\begin{definition}
Given a function $B:\R^n\to \R$ and disjoint sets $S,T\subset [n]$ we write $B_{S;T}:\R^T\to\R$ to be the function 
$$B_{S;T}(\vec s_T):=\int_{\R^R}B(\vec 0_S,\vec s_T,\vec s_R)d\vec s_R$$ for $R=(S\cup T)^c$. In other words, we restrict the $S$ co-ordinates to be $0$ and integrate over the remaining co-ordinates except for $T$. We also write
$B_S$ to denote $B_{S;\emptyset}$.
\end{definition}

\subsection{The part with $t$ small}

In this section, we evaluate $\I^{(1)}(\phi^\gb, \cB;X)$:
\begin{proposition}\label{prop:I1}
We have
\begin{equation*}
\begin{array}{rcl}
\I^{(1)}(\phi^\gb,\cB;X)&=&\displaystyle
X\Vol_{dg_3}(\FF_3)\wt{\eta}(12)\nu(\phi)\Vol(\cB)\int_{g_2\in\FF_2}f_0^X(g_2)dg_2+C_{00}(\phi) X^{5/6}
\\[.2in]&&
+cX^{21/24}+O(X^{23/28+3\delta+o(1)})
\end{array}
\end{equation*}
for some constants $c=c_\phi$ and $C_{00}(\phi)$.
\end{proposition}

The proof of this result closely follows the lines of the proof of \cite[Proposition 6.1]{ST_second_main_term_1}, but there are some important differences. These stem from the following issues that arise in our situation here:
\begin{enumerate}
\item In \cite[Proposition 6.1]{ST_second_main_term_1}, the sum was over a lattice $L$ (with no further imposed congruence conditions), while in our situation we are summing $\phi^\gb$ over $V(\Z)$;
\item The lattice $L$ is assumed in \cite[Proposition 6.1]{ST_second_main_term_1} to be an $S_4$-lattice. This meant that $L$ contained no non-generic elements. In our situation, this is not true. Instead, $\phi^\gb$ is supported away from the $\det(A)=0$ and $a_{11}=b_{11}=0$ slices, which we have to address.
\item The assumption of $L$ being an $S_4$-lattice is additionaly used to deduce that $\nu(L_{\{a_{11}=b_{11}=0})=\vec\zeta_L(\{a_{11},b_{11}\}=0)$ is $0$. We have to be more careful in our situation - the analogous statement is not true when working with $\phi$.
\end{enumerate}
We overcome these issues as follows. First, we note that $\phi$ is defined via congruence conditions modulo some positive integer, and hence is a weighted union of lattices. The characteristic function of $V(\Z)^\gb$ is however not of this form - it excludes the $a_{11}=b_{11}=0$ and $\det(A)=0$ slices. Missing the $a_{11}=b_{11}=0$ slice will not be an issue for us. Indeed, the lattice structure of $L$ is never directly used in the proof of \cite[Proposition 6.1]{ST_second_main_term_1}. Rather it is used for the slices of $L$ with $a_{11}$ and $b_{11}$ fixed. We will also slice over values of $a_{11}$ and $b_{11}$, and exclude by hand the $a_{11}=b_{11}=0$ slice. To deal with the $\det(A)=0$ slice, we apply Proposition \ref{prop_reducible_bound_resolvent}, which implies that we have
\begin{equation*}
|\cI^{(i)}(\phi^\gb,\cB;X)-\cI^{(i)}(\phi^{\gb'},\cB;X)|\ll X^{23/28+3\delta+o(1)},
\end{equation*}
where $\phi^{\gb'}$ is the product of $\phi$ with the characteristic function of the set of elements in $V(\Z)$ with $(a_{11},b_{11})\neq (0,0)$. It thus suffices to work with the function $\phi^{\gb'}$ instead of $\phi^{\gb}$, sidestepping the issue of the $\det(A)=0$ elements. Moreover, since we will immediately be slicing over $a_{11}$ and $b_{11}$, we can simply work with $\phi$, as long as we avoid the $a_{11}=b_{11}=0$ slice.

We must also remember that the count above includes nongeneric points (aside from those with $a_{11}=b_{11}=0$), and we deal with them in \S6. Finally, for Issue (c), we will proceed as in the proof of \cite[Proposition 6.1]{ST_second_main_term_1}, keeping track of when terms containing these zeta functions are discarded. Here, instead of discarding the terms, we will evaluate them.

\medskip

\begin{proof}
We begin with the following lemma.
\begin{lemma}
Set $S:=\{a_{11},b_{11}\}$. Then we have
\begin{equation*}
\begin{array}{ll}
&\displaystyle
\left |\int_{g\in\FF}\sum_{x\in V(\Z)} \phi^{\gb'}(x)(g\cB)(x)\eta\Big(\frac{\lambda(g)}{X^{1/12}}\Big)f_0^X(g_2)dg\right.
\\[.2in]
&\displaystyle-\left.\int_{g\in\FF}\sum_{\substack{a_{11},b_{11}\\(a_{11},b_{11})\neq(0,0)}}\nu(\phi|_{S})(g\cB)_S(a_{11},b_{11})\eta\Big(\frac{\lambda(g)}{X^{1/12}}\Big)f_0^X(g_2)dg \right |
\ll  X^{3/4+O(\delta)}.
\end{array}
\end{equation*}
\end{lemma}
\noindent The proof of the above lemma is identical to that of \cite[Corollary 6.3]{ST_second_main_term_1}, noting only that in both cases we are not counting lattice points with $a_{11}=b_{11}=0$. This yields
\begin{equation}\label{eq:cor_t_small_1}
\begin{array}{ll}
\I^{(1)}(\phi^{\gb'},\cB;X)=&\displaystyle
\int_{\lambda>0}\int_{s_1,s_2}\int_{g_2\in\FF_2}\sum_{a_{11},b_{11}}\nu(\phi|_{S})(\lambda\cdot (s_1,s_2)g_2)\cB)_S(a_{11},b_{11})
\\[.2in]
&\displaystyle\quad\quad \eta\Big(\frac{\lambda}{X^{1/12}}\Big)f_0^X(g_2)\delta_{\FF_3}(s_1,s_2)\frac{d^\times sd^\times\lambda dg_2}{s_1^6s_2^6} + O(X^{3/4 + O(\delta)}).
\end{array}
\end{equation}
We now proceed as in \cite[\S6.2]{ST_second_main_term_1}, perform a Mellin transform, and divide the sum in the right hand side into a sum of three terms: corresponding to $\{a_{11}b_{11}\neq 0\}$, $\{a_{11}=0,b_{11}\neq 0$\}, and $\{a_{11}\neq 0,b_{11}=0\}$. As in \cite[\S6.2]{ST_second_main_term_1}, the sum of these three terms yields the main term (which is equal to that in the statement of our proposition), a term of size $X^{21/24}$ with an undetermined leading constant (denoted by $c$ in the statement of Proposition \ref{prop:I1}), an error term $O(X^{3/4+o(1)})$, and one other term, namely,
\begin{equation}\label{eq:extra561}
\int_{\lambda>0}\lambda^{10}\eta\Big(\frac{\lambda}{X^{1/12}}\Big)d^{\times}\lambda\int_{s_1,s_2}\int_{g_2\in\FF_2}\wt{g_2\cB}_{\{a_{11},b_{11}\}}(0)\vec{\zeta}_\phi(\{a_{11},b_{11}\}=0)\wt{\delta}_{\FF_3}(2,-2)\frac{d^\times s dg_2}{s_1^6s_2^6}.
\end{equation}
Note that in \cite[\S6.2]{ST_second_main_term_1}, the term analogous to the above was discarded since there the value $\vec{\zeta}_L(\{a_{11},b_{11}\}=0)$ is $0$ ($L$ being an $S_4$-lattice). In our situation, this gives an extra term of size $X^{5/6}$ with leading constant that we denote by $C_{00}(\phi)$. This extra term being accounted for in the statement of the proposition, the result follows.
\end{proof}

The following lemma is necessary for the last claim of Theorem \ref{thm_point_count_generic}.
\begin{lemma}\label{lem_C00_q}
Let $q$ be a squarefree integer relatively prime to the integer modulo which $\phi$ is defined. Let $\phi_q$ denote the product of $\phi$ with the characteristic function of the set of elements in $V(\Z)$ which are nonmaximal at every prime dividing $q$. Then we have
\begin{equation*}
C_{00}(\phi_q)=C_{00}(\phi)\Bigl(\frac1{q}+O(q^{-2})\Bigr).
\end{equation*}
\end{lemma}
\begin{proof}
It is only necessary to compute $\vec\zeta_{\phi_q}(\{a_{11},b_{11}\})$, the density in $V(\Z)_{a_{11}=b_{11}=0}$ of $\phi_q$. By assumption on $q$, this is simply the product of the analogous densities of $\phi$ and of nonmaximality at $q$. The latter density has been estimated in Lemma \ref{lem_density_nonmax_00} to be $1/q+O(1/q^2)$, completing the proof.
\end{proof}

\subsection{The part with $t$ large}
In this section, we evaluate $\I^{(2)}(\phi^\gb, \cB;X)$:
\begin{proposition}\label{prop:section6_t_big_main}
We have
\begin{equation*}
\begin{array}{rcl}
\I^{(2)}(\phi^\gb,\cB;X)&=&
\displaystyle X\Vol_{dg_3}(\FF_3)\wt{\eta}(12)\nu(\phi)\Vol(\cB)\int_{g_2\in\FF_2}f_0^X(g_2)dg_2
\\[.15in]
&&+C\wt{\eta}(10)Z\bigl((\cB)_{\emptyset;\{A\}},\nu(\phi_A);4/3\bigr)X^{5/6}
\\[.15in]&&\displaystyle
+c_0X^{21/24+5\delta/2}+c_1X^{21/24-3\delta}+c_2X^{21/24-3\delta/2}+O(X^{5/6-\theta'}),
\end{array}
\end{equation*}
for some constants $C$, $c_0$, $c_1$, $c_2$, and $\theta'>0$, where $Z((\cB)_{\emptyset;{A}}\nu(\phi_A);s)$ is the global zeta integral associated to the space of ternary quadratic forms.
\end{proposition}
\begin{proof}
Similarly to the previous subsection, we reduce to the case where the methods and results from \cite[\S7]{ST_second_main_term_1} apply. First, we note that
the proof of \cite[Lemma 7.2]{ST_second_main_term_1} applies verbatim to our situation, yielding the following analogue of \cite[(26)]{ST_second_main_term_1}:
\begin{equation}\label{eq:temp_t_large_1}
\I^{(2)}(\phi^\gb,\cB;X)=\int_{g\in\FF}\sum_{\substack{A,b_{11}\\\det(A)\neq 0\\(a_{11},b_{11})\neq 0}}\nu(\phi|_{\{A,b_{11}\}})(g\cB)_{\emptyset;\{A,b_{11}\}}(A,b_{11})
\eta\Big(\frac{\lambda(g)}{X^{1/12}}\Big)f^X(g)dg + O(X^{3/4 + O(\delta)}).
\end{equation}
Second, we have the following lemma.
\begin{lemma}\label{lemma_5.8}
We have
\begin{equation}\label{eq:temp_t_large_2}
\begin{array}{rcl}
\displaystyle\I^{(2)}(\phi^{\gb},\cB;X)
&=&\displaystyle
\int_{g\in\FF}\sum_{\substack{A\\\det(A)\neq 0}}\nu(\phi_{A})(g\cB)_{\emptyset;\{A\}}(A)\eta\Big(\frac{\lambda(g)}{X^{1/12}}\Big)f^X(g)dg
\\[.2in]&&\displaystyle 
+c_1X^{21/24-3\delta}+c_2X^{21/24-3\delta/2}
+O(X^{5/6-\theta'}),
\end{array}
\end{equation}
for some constants $c_1$, $c_2$, and $\theta'>0$.
\end{lemma}
\noindent The analogue of the above lemma is proved in \cite[Proposition 7.3]{ST_second_main_term_1}. The modifications of that proof required for us are somewhat substantial, so we prove it in some detail.

\medskip

\begin{proof}
We start by defining $E^{(1)}_g$ and $E^{(2)}_g$ to be
\begin{equation*}
\begin{array}{rcl}
E^{(1)}_g&:=&\displaystyle\sum_{\substack{A,b_{11}\\\det(A)\neq 0\\(a_{11},b_{11})\neq (0,0)}}\nu(\phi|_{A,b_{11}})(g\cB)_{\emptyset;\{A,b_{11}\}}-\sum_{\substack{A\\\det(A)\neq 0}}\nu(\phi_{A})(g\cB)_{\emptyset;\{A\}}(A);
\\[.15in]
E^{(2)}_g&:=&\displaystyle\sum_{b_{11}\neq 0}\nu(\phi|_{a_{11}=0;b_{11}})(g\cB)_{\{a_{11}\};\{b_{11}\}}-\nu(\phi_{a_{11}=0})(g\cB)_{\{a_{11}\}}.
\end{array}
\end{equation*}
Note that in light of \eqref{eq:temp_t_large_1}, to prove the lemma, it is enough to evaluate the integral of $E_g^{(1)}$ against $\eta(\lambda(g)/X^{1/12})f^X(g)dg$. As in the proof of \cite[Proposition 7.3]{ST_second_main_term_1}, we do this by first proving that the integral of $E_g^{(1)}-E_g^{(2)}$ is small, and then evaluating the integral of $E_g^{(2)}$. 
\begin{lemma}\label{lem:Eg1isEg2}
We have
\begin{equation*}
\int_{g\in\FF}(E_g^{(1)}-E_g^{(2)})\eta\Big(\frac{\lambda(g)}{X^{1/12}}\Big)f^X(g)dg\ll X^{5/6-\theta'}
\end{equation*}
for some $\theta'>0$.
\end{lemma}
\begin{proof}
We prove the result by dividing $\FF$ up into pieces. Let $\theta$ be a small positive real number.

\medskip

\noindent{\bf The piece $\FF':=\{g\in\FF:\;\lambda(g)\asymp X^{1/12},\;f^X(g)>0,\;w_g(b_{11})>X^\theta\}$:}
The restriction of the integral of $E_g^{(2)}-E_g^{(1)}$ to this region is equal to - up to a super-polynomially decaying term- the following:
\begin{equation*}
\int_{g\in\FF'}\Bigl(
\sum_{\substack{A:a_{11}=0\\\det(A)\neq 0}}\nu(\phi|_{a_{11}=b_{11}=0;A\backslash\{a_{11}\}})(g\cB)_{\{a_{11},b_{11}\};\{A\backslash\{a_{11}\}\}}-
\nu(\phi|_{a_{11}=b_{11}=0})(g\cB)_{\{a_{11},b_{11}\}}\Bigr)dg.
\end{equation*}
The integrand above is superpolynomially small when $w_g(a_{12})>X^{\theta_1}$ for any $\theta_1>0$. We may hence further restrict the integral to the region $g\in\FF'$ with $w_g(a_{12})<X^{\theta_1}$. Every such $g$ satisfies
\begin{equation*}
\lambda ts_1^{-4}s_2^{-2}\gg X^\theta,\quad
\lambda t^{-1}s_1^{-1}s_2^{-2}\ll X^{\theta_1}\implies
t^2s_1^{-3}X^{\theta_1-\theta}\gg 1.
\end{equation*}
It is then easy to see that the integral over $\FF'$ above is $\ll$
\begin{equation*}
\int_{\lambda\asymp X^{1/12}}\int_{\substack{t,s_1,s_2\\w(b_{11})>X^\theta\\w(a_{12})<X^{\theta_1}}}
(\lambda^{10}t^{-2}s_1^{-2}s_2^{-2}+\lambda^9t^{-1}s_1^3)d^\times\lambda d^\times t d^\times s_1d^\times s_2\ll X^{5/6-\theta_2},
\end{equation*}
for some sufficiently small $\theta_2$ as long as $\theta_1$ is chosen to be smaller than $\theta$.

\medskip
\noindent{\bf The piece $\FF''=\{g\in\FF:\lambda(g)\asymp X^{1/12},\;f^X(g)>0,\;w_g(b_{11})<X^\theta,\;w_g(a_{12})>X^{\theta_3}\}$:}
Here, $\theta_3$ is some small positive real number. We write
$E_g^{(1)}-E_g^{(2)}=F_g^{(1)}-F_g^{(2)}$, where
\begin{equation*}
\begin{array}{rcl}
F^{(1)}_g&:=&\displaystyle\sum_{\substack{A,b_{11}\\\det(A)\neq 0\\(a_{11},b_{11})\neq (0,0)}}\nu(\phi|_{A,b_{11}})(g\cB)_{\emptyset;\{A,b_{11}\}}-\sum_{b_{11}\neq 0}\nu(\phi|_{a_{11}=0;b_{11}})(g\cB)_{\{a_{11}\};\{b_{11}\}};
\\[.15in]
F^{(2)}_g&:=&\displaystyle \sum_{\substack{A\\\det(A)\neq 0}}\nu(\phi_{A})(g\cB)_{\emptyset;\{A\}}(A)-\nu(\phi_{a_{11}=0})(g\cB)_{\{a_{11}\}}.
\end{array}
\end{equation*}
The bound on the integral of $F_g^{(1)}-F_g^{(2)}$ over $\FF''$ follows just as in the proof of \cite[Proposition 7.3]{ST_second_main_term_1}, with one caveat:
it is also necessary to bound the contribution from the $\det(A)=0$ in the two summands of $E_g^{(1)}$. These contributions are bounded immediately by applying \cite[Theorem 5.1]{BSW_sqfree2}, from which it follows that the number of choices for such $A$ are $O(\lambda^3t^{-3}s_1^6s_2^6)$. With this estimate, and keeping in mind that the condition $b_{11}\neq 0$ implies we can restrict to $g$ with $w_g(b_{11})\gg 1$, we obtain the bound $\lambda^3t^{-3}s_1^6s_2^6\cdot \lambda^6t^6=\lambda^9t^3s_1^6s_2^6$ integrated over this region of the fundamental domain. The bound follows since we have
\begin{equation*}
\int_{\substack{t,s_1,s_2\\w(a_{12})>X^{\theta_3}}} \lambda^9 t d^\times t d^\times s_1 d^\times s_2\ll \lambda^{10-\theta_3},
\end{equation*}
which is sufficiently small.
\medskip

\noindent{\bf The final piece $\FF\backslash(\FF'\cup \FF'')$:}
This piece consists of  $g\in\FF$ satisfying $\lambda(g)\asymp X^{1/12}$, $f^X(g)>0$, $w_g(b_{11})\ll X^\theta$, and $w_g(a_{12})\ll X^{\theta_3}$. The required bound follows identically as in the proof of \cite[Proposition 7.3]{ST_second_main_term_1}, where the integral of each of the four summands of $E_g^{(1)}$ and $E_g^{(2)}$ over this region in $\FF$ was shown to be sufficiently small.
\end{proof}

Next, we have the following lemma whose proof is identical to the proof of \cite[Lemma 7.5]{ST_second_main_term_1}:

\begin{lemma}
We have
\begin{equation*}
\int_{g\in\FF} E_g^{(2)}
\psi\Big(\frac{\lambda(g)}{X^{1/12}}\Big)f^X(g)dg=c_1X^{21/24-3\delta}+O(X^{3/4}).
\end{equation*}
\end{lemma}

Combining the two lemmas above, we see that Lemma \ref{lemma_5.8} follows.
\end{proof}

\noindent We continue with the proof of Proposition \ref{prop:section6_t_big_main}. From \eqref{eq:temp_t_large_2}, the methods of \cite[\S7.3]{ST_second_main_term_1} go through without change, yielding
\begin{equation*}
\begin{array}{rcl}
\I^{(2)}(\phi^\gb,\cB;X)&=&
\displaystyle C\wt{\eta}(12)\Res_{s=2}Z((\cB)_{\emptyset;{A}},\nu(\phi_A);s)\wt{f^X}(-2)X
\\[.15in]
&&+C\wt{\eta}(10)Z\bigl((\cB)_{\emptyset;\{A\}},\nu(\phi_A);4/3\bigr)X^{5/6}
\\[.15in]&&\displaystyle
+c_0X^{21/24+5\delta/2}+c_1X^{21/24-3\delta}+c_2X^{21/24-3\delta/2}+O(X^{5/6-\theta'}),
\end{array}
\end{equation*}
for some constants $C$, $c_0$, $c_1$, and $c_2$, where $Z((\cB)_{\emptyset;{A}}\nu(\phi_A);s)$ is the global zeta integral associated to the space of ternary quadratic forms. Finally, the fact that the leading constant of $X$ is as claimed in the proposition follows from computations in \cite[\S8]{ST_second_main_term_1}, which apply to our setting without change.
\end{proof}

Theorem \ref{thm_point_count_generic} follows from Propositions \ref{prop:I1} and \ref{prop:section6_t_big_main} by noting that $\delta$ can be arbitrarily picked, and so any terms whose exponents depend on $\delta$ must necessarily vanish.

\section{Putting it all together}

In this section, we prove our main results, Theorems \ref{th_main_counting_result}, \ref{thm:mainS4rings}, and \ref{th_shintani_residue_main}. We do it in reverse order, beginning with the residue at $5/6$ of our Shintani zeta functions.

\medskip

\noindent{\bf Proof of Theorem \ref{th_shintani_residue_main}:} Let $\phi:V(\Z)\to\R$ be a periodic $G(\Z)$-invariant function, and let $i\in\{0,1,2\}$ be fixed. General theory due to Sato--Shintani \cite{SatoShintani} implies that the Shintani zeta functions $\xi_i(\phi;s)$ only has possible poles at $s=1$, $5/6$, and $3/4$. Let $\cB:V(\R)^{(i)}\to\R$. By Proposition \ref{prop_global_zeta_shintani_G}, the same must then be true for the global zeta $Z(\phi,\cB;s)$, for any function $\cB:V(\R)^{(i)}\to\R$ with compact support away from the discriminant-$0$ locus. In fact, the proposition implies that $Z(\phi,\cB;s)$ and $\xi_i(\phi;s)$ have poles of the same order at $s=1$, $5/6$, and $3/4$. Let $\psi:\R_{\geq 0}\to\R$ be any function with compact support. We will prove that $I_\psi(\phi,\cB;X)$ has a power saving expansion consisting of terms of size $X\log X$, $X$, and $X^{5/6}$, up to an error of $O(X^{5/6-\theta})$ for some positive real number $\theta$. In conjunction with \eqref{eq:I_to_Z_Mellin}, this will yield Theorem \ref{th_shintani_residue_main}.

We pick $\cB$ to be of the form \eqref{eq:good_cB}: namely, we let $\eta,\eta_1:\R_{>0}\to\R_{\geq 0}$ be a smooth function of compact support, and we define $\cB:V(\R)^{(i)}\to\R$ by setting 
\begin{equation}
\cB(x)=\sum_{\substack{g=\lambda g_1\in G(\R)\\g y=x}}\mathcal{G}(g_1)\eta_1(\lambda).
\end{equation}

Once again, we pick $\psi$ so that $$\wt{\psi}(s)=\wt{\eta}(12s)\wt{\eta_1}(12s)$$

and write
\begin{equation}\label{eq_I_phi_form}
\begin{array}{rcl}
\I_\eta(\phi,\cB;X)&=&
\I_\eta(\phi^\gen,\cB;X)+\I_\eta(\phi^\ngen,\cB;X)
\\[.1in]
&=&\I_\eta(\phi^{\gb},\cB;X)+\I_\eta(\phi^{\ngen},\cB;X)-\I_\eta(\phi^{\ngen,\gb},\cB;X),
\end{array}
\end{equation}
where the auxiliary functions are defined as follows. 
\begin{enumerate}
\item $\phi^\gen$ is the product of $\phi$ with the characteristic function of generic elements in $V(\Z)$;
\item $\phi^\ngen$ is the product of $\phi$ with the characteristic function of non-generic elements in $V(\Z)$;
\item $\phi^{\gb}$ is the product of $\phi$ with the characteristic function of elements in the generic body in $V(\Z)$;
\item $\phi^{\ngen,\gb}$ is the product of $\phi$ with the characteristic function of non-generic elements in the generic body in $V(\Z)$.
\end{enumerate}
Note that the latter two functions are not even $G(\Z)$-invariant, but $\cI_\eta(\cdot,\cB;X)$ can be defined simply using the original definition of $\cI$. Power saving error terms for $\I_\eta(\phi^{\gb},\cB;X)$, $\I_\eta(\phi^{\ngen},\cB;X)$, and $\I_\eta(\phi^{\ngen,\gb}\cB;X)$ are given in Theorem \ref{thm_point_count_generic}, Corollary \ref{cor:counting_red_rings_I}, and Theorem \ref{th_red_main_body_final}, respectively, with main terms of magnitude $X\log X$, $X$, and $X^{5/6}$, up to power saving error terms. Theorem \ref{th_shintani_residue_main} now follows from the absence of an $X^{5/6}\log X$ term in any of these three results. $\Box$

\medskip

We turn now to the proof of Theorem \ref{thm:mainS4rings}. Let $\Lambda=(\Lambda_v)_{v\in S}$ be a finite collection of local specifications for quartic rings. Let $\wt{R}(\Lambda)$ denote the set of all quartic triples, with nonzero discriminant, whose base change to $\Z_v$ lies in $\Lambda_v$ for all $v\in S$. (Recall that $R(\Lambda)$ is the set of triples $(Q,C,r)\in \wt{R}(\Lambda)$ such that $Q$ is an order in an $S_4$-field.) 
Bhargava's parametrization \cite{BHCL3} gives a bijection between $\wt{R}(\Lambda)$ and the set of $G(\Z)$-orbits on some set $L_\Lambda\subset V(\Z)$. Since $\Lambda$ is assumed to be finite, the set $L_\Lambda$ must be defined by finitely many congruence conditions (modulo powers of the primes in $S$). Let $\phi_\Lambda$ denote the characteristic function of $L_\Lambda$. Without loss of generality, we assume that $S$ contains the infinite place of $\Q$, and that $\Lambda_v$ is the singleton set $\R^{4-2i}\times\C^i$ for some $i\in\{0,1,2\}$.
To prove Theorem \ref{thm:mainS4rings}, we need a closer analysis of the leading constants of the power saving expansion of $\I_\eta(\phi_\Lambda,\cB;X)$. To this end, we have the following result.
\begin{lemma}\label{lem:constant_cancel}
Let $\phi:V(\Z)\to \R$ be a periodic $G(\Z)$-invariant function. Let $C_\phi^{\ngen,\gb}$ and $C_\phi^{00}$ be as in Theorem \ref{th_red_main_body_final} and Theorem \ref{thm_point_count_generic}, respectively. We have $C_\phi^{\ngen,gb}=C_\phi^{00}$.
\end{lemma}
\begin{proof}
Let $N$ denote the integer modulo which $\phi$ is defined, and let $q$ be a squarefree integer relatively prime to $N$. We let $\phi_q$ denote the product of $\phi$ with the characteristic function of the set of elements in $V(\Z)$ which are nonmaximal at every prime dividing $q$. As in the proof of Theorem \ref{th_shintani_residue_main}, we combine \eqref{eq_I_phi_form}, Theorems \ref{thm_red_points_count} and \ref{thm_point_count_generic}, and Corollary \ref{cor:counting_red_rings_I} to obtain a power series expansion for $\I_\eta(\phi_q,\cB;X)$. From Proposition \ref{prop:I_to_N}, we then obtain a power series expansion for $N_{\psi}^{(i)}(\phi_q;X)$. We will now study how the leading constant of the $X^{5/6}$ term in this latter power saving expansion behaves as $q$ varies. This constant, denoted by $c(\phi,q)$, is easily computed from the previously referenced to be
\begin{equation*}
c(\phi,q)=\frac{12}{C_{\mathcal G}A_i}\Bigl(
c_{5/6,q}^\red+(C_{\phi}^{\ngen,\gb}-
C_{00}(\phi))\bigl(1/q+O(q^{-2}\bigr)+C\wt{\phi}(10)Z\bigl((\cB)_{\emptyset;\{A\}},\nu(\phi_{q,A});4/3\bigr)
\Bigr),
\end{equation*}
for some absolute constants $C_{\mathcal G}$, $A_i$, and $C$.

We now use our crucial input, namely, \cite[Proposition 12.4]{ST_second_main_term_1}, which implies that the sum of $c(\phi,q)$ over squarefree $q$ converges absolutely. The final claim of Theorem \ref{thm_red_points_count} implies that the sum over $q$ of $c_{5/6,q}^\red$ converges absolutely. The constant $Z\bigl((\cB)_{\emptyset;\{A\}},\nu(\phi_{q,A});4/3\bigr)$ has been evaluated in \cite[Corollary 8.4]{ST_second_main_term_1} to be
\begin{equation*}
Z\bigl((\cB)_{\emptyset;\{A\}},\nu(\phi_{q,A});4/3\bigr)
=\zeta_{\nu(\phi_{q,A})}(4/3)\mathcal W_{2/3}(\cB)
-
\zeta_{\nu(\phi_{q,A})'}(4/3)\mathcal W_{2/3}'(\cB).
\end{equation*}
The constants $\zeta_{\nu(\phi_{q,A})}(4/3)$ and $\zeta_{\nu(\phi_{q,A})'}(4/3)$ are computed in \cite[Equations (39),(40)]{ST_second_main_term_1}, and it is easy to check from those computations that the sum of $Z\bigl((\cB)_{\emptyset;\{A\}},\nu(\phi_{q,A});4/3\bigr)$ over $q$ also converges absolutely. Therefore, it follows that the sum over $q$ of
\begin{equation*}
(C_{\phi}^{\ngen,\gb}-
C_{00}(\phi))\bigl(1/q+O(q^{-2}\bigr)
\end{equation*}
converges absolutely, from which we obtain the desired result.
\end{proof}

We obtain the following consequence of the above result.
\begin{corollary}\label{cor:I_improved_error}
Let $\phi:V(\Z)\to\R$ be a periodic $G(\Z)$-invariant function. Then we have
\begin{equation*}
\begin{array}{rcl}
\I_\eta(\phi^\gb,\cB;X)-\cI_\eta(\phi^{\ngen,\gb},\cB;X)&=&
\displaystyle\Vol_{dg}(\FF)\wt{\eta}(12)\nu(\phi)\Vol(\cB)X
\\[.15in]
&&+C\wt{\eta}(10)Z\bigl((\cB)_{\emptyset;\{A\}},\nu(\phi_A);4/3\bigr)X^{5/6}+O(X^{3/4+o(1)}),
\end{array}
\end{equation*}
where $C$ is as in Theorem \ref{thm_point_count_generic}.
\end{corollary}
\begin{proof}
We begin by writing
\begin{equation*}
\I_\eta(\phi^\gb,\cB;X)-\cI_\eta(\phi^{\ngen,\gb}\cB;X)=\cI_\eta(\phi,\cB;X)-\I_\eta(\phi^\ngen,\cB;X).
\end{equation*}
Combining Theorems \ref{th_red_main_body_final} and \ref{thm_point_count_generic} with Lemma \ref{lem:constant_cancel}, we see that  the leading constants for the power saving expansion of the left hand side of the above equation are as claimed, up to a power saving error. It is thus only required to bound the error term by $O(X^{3/4+o(1)})$. For this, note that the power series expansion of 
$I_\eta(\phi,\cB;X)$ has an error term of $O(X^{3/4+o(1)})$ by using \eqref{eq:I_to_Z_Mellin}, and noting that the location of the poles of $Z(\phi,\cB;s)$ are exactly at the poles of $\xi_i(\phi;s)$, namely at $1$, $5/6$, and $3/4$. Combining this with Corollary \ref{cor:counting_red_rings_I}, we see that the power saving expansion of the right hand side of the above displayed equation has at most terms of magnitude $X\log X$, $X$, and $X^{5/6}$, up to an error of $O(X^{3/4+o(1)})$. The claim therefore follows.
\end{proof}

 Analogously to the definition of $N_\Lambda(\psi,X)$ in \eqref{eq:rings_count}, we define
\begin{equation}\label{eq:all_rings_count}
\wt{N}_{\Lambda}(\psi,X):=\sum_{(Q,C,r)\in \wt{R}(\Lambda)}\frac{1}{|\Aut(Q,C,r)|}\psi\Bigl(\frac{|\Delta(Q)|}{X}\Bigr).
\end{equation}
Then we have
\begin{equation}\label{eq:N_tilde_def}
\wt{N}_\Lambda(\psi,X)=N^{(i)}_\psi(\phi_\Lambda;X).
\end{equation}
We are now ready to prove our main result on counting rings.

\medskip

\noindent{\bf Proof of Theorem \ref{thm:mainS4rings}:} 
We begin by evaluating $\wt{N}_\Lambda(\psi,X)$. From \eqref{eq:N_tilde_def} and \eqref{eq:N_to_xi_Mellin}, we obtain
\begin{equation*}
\begin{array}{rcl}
\wt{N}_\Lambda(\psi,X)&=&\displaystyle
\int_2\xi_i(\phi_\Lambda;s)X^s\wt{\psi}(s)ds
\\[.2in]&=&\displaystyle
\wt{\psi}(1)r_2(\phi_\Lambda;1)X\log X +
\wt{\psi}(1)r(\phi_\Lambda;1)X +
\wt{\psi}(5/6)r(\phi_\Lambda;5/6)X^{5/6}+O(X^{3/4+o(1)}).
\end{array}
\end{equation*}
Therefore, to prove Theorem \ref{thm:mainS4rings}, it suffices to evaluate the constants $r_2(\phi_\Lambda;1)$, $r(\phi_\Lambda;1)$, $r(\phi_\Lambda;5/6)$, and then subtract from $\wt{N}_\Lambda(\psi,X)$ the contribution from non $S_4$-rings (or up to an error of $O(X^{.77})$, nongeneric rings). We note two things: first, to determine the values of these three constants, it suffices to evaluate $\wt{N}_\Lambda(\psi,X)$ for any function $\psi$; and second, in Theorem \ref{thm_red_points_count}, we have evaluated the contribution from nongeneric rings as a power series, with terms $X\log X$, $X$, and $X^{5/6}$, up to an error of $O(X^{3/4+o(1)})$.

Applying Proposition \ref{prop:I_to_N}, we obtain
\begin{equation*}
\begin{array}{rcl}
\wt{N}_\Lambda(\psi,X)=N_{\psi}^{(i)}(\phi_\Lambda;X)
&=&\displaystyle\frac{12}{C_{\mathcal G}A_i}I_{\eta}(\phi_\Lambda,\cB;X);
\\[.15in]
{N}_\Lambda(\psi X)=N_{\psi}^{(i)}(\phi_\Lambda^\gb-\phi_\Lambda^{\ngen,\gb};X)
&=&\displaystyle\frac{12}{C_{\mathcal G}A_i}\bigl(I_{\eta}(\phi_\Lambda^\gb,\cB;X)-I_{\eta}(\phi_\Lambda^{\ngen,\gb},\cB;X)\bigr).
\end{array}
\end{equation*}
In Corollary \ref{cor:I_improved_error}, we have obtained a power series expansion for $I_{\eta}(\phi_\Lambda^\gb,\cB;X)-I_{\eta}(\phi_\Lambda^{\ngen,\gb},\cB;X)$, with an error of $O(X^{3/4+o(1)})$. It is therefore only necessary to verify that the leading constants of the $X$ and $X^{5/6}$ terms of this expansion, multiplied by $12/(C_{\mathcal G}A_i)$, are respectively equal to the leading constants of the $X$ and $X^{5/6}$ terms in Theorem \ref{thm:mainS4rings}. This follows in identical fashion to the case when $\Lambda$ is an $S_4$-family, proved in \cite[\S 8]{ST_second_main_term_1}. $\Box$

\medskip

Finally, we turn to the proof of our main fields counting result.

\medskip

\noindent{\bf Proof of Theorem \ref{th_main_counting_result}:} Let $\Sigma=(\Sigma_v)_{v\in S}$ be a finite collection of local specifications for quartic fields. Let $\wt{F}(\Sigma)$ denote the set of all \'etale quartic algebras satisfying the local specifications of $\Sigma$. (Recall that $F(\Sigma)$ is the set of quartic fields in $\wt{F}(\Sigma)$.) For a smooth function $\psi:\R_{\geq 0}\to\R$ with compact support, we define
\begin{equation}\label{eq:fields_count}
\wt{N}_{\Sigma}(\psi,X):=\sum_{K\in \wt{F}(\Sigma)}\psi\Bigl(\frac{|\Delta(K)|}{X}\Bigr).
\end{equation}
In \cite[(53)]{ST_second_main_term_1}, we prove that we have
\begin{equation}\label{eq:sieve_input}
\wt{N}_{\Sigma}(\psi,X)=\sum_{c\in\{1,5/6\}} \bigl(C_2(c,\psi,\Sigma)\log X+C_1(c,\psi\Sigma)\bigr)X^c+O(X^{13/16+o(1)}),
\end{equation}
for some constants $C_2(c,\psi,\Sigma)$ and $C_1(c,\psi,\Sigma)$.
Moreover (and crucially), these constants can be described by the following limits. Let $\Lambda=(\Lambda_v)_{v\in S}$ be the finite collection of local specifications on quartic rings associated to $\Sigma$. Let $\phi_\Lambda:V(\Z)\to\R$ be the $G(\Z)$-invariant function corresponding to $\Lambda$ as above (note that $\phi$ is periodic). For a squarefree integer $q$ relatively prime to the primes in $S$, let $\phi_{\Lambda,q}$ be the product of $\phi_\Lambda$ with the characteristic function of the set of elements nonmaximal at $q$. Write
\begin{equation*}
N^{(i)}_{\psi}(\phi_{\Lambda_q};X)=\sum_{c\in\{1,5/6\}} \bigl(C_2(c,\psi,\Lambda_q)\log X+C_1(c,\psi\Lambda_q)\bigr)X^c+O(X^{3/4+o(1)}).
\end{equation*}
Then the discussion following \cite[(52)]{ST_second_main_term_1} implies that we have
\begin{equation*}
\begin{array}{rcl}
C_2(c,\psi,\Sigma)=\displaystyle\sum_{q}\mu(q)C_2(c,\psi,\Lambda_q);\quad
C_1(c,\psi,\Sigma)=\sum_{q}\mu(q)C_1(c,\psi,\Lambda_q).
\end{array}
\end{equation*}
Therefore, to prove Theorem \ref{th_main_counting_result}, it only remains to compute the constants $C_2(c,\psi,\Sigma)$ and $C_1(c,\psi,\Sigma)$ via these limits, and then subtract from $\wt{N}_\Sigma(\psi,X)$ the contribution of \'etale algebras which are not $S_4$ quartic fields. We do these two steps simultaneously.

For $i\in\{1,2\}$, write $C_i(c,\psi,\Lambda_q)$ as a sum $C_i^\gen(c,\psi,\Lambda_q)+C_i^\ngen(c,\psi,\Lambda_q)$, where these two summands are the leading constants of the power series expansion of $N_\psi^{(i)}(\phi_{\Lambda_q}^\gen;X)$ and $N_\psi^{(i)}(\phi_{\Lambda_q}^\ngen;X)$, respectively. That is, we have
\begin{equation*}
\begin{array}{rcl}
N^{(i)}_{\psi}(\phi_{\Lambda_q}^\gen;X) &=&\displaystyle
\sum_{c\in\{1,5/6\}} \bigl(C_2^\gen(c,\psi,\Lambda_q)\log X+C_1^\gen(c,\psi\Lambda_q)\bigr)X^c+O(X^{3/4+o(1)});
\\[.15in]
N^{(i)}_{\psi}(\phi_{\Lambda_q}^\ngen;X) &=&\displaystyle
\sum_{c\in\{1,5/6\}} \bigl(C_2^\ngen(c,\psi,\Lambda_q)\log X+C_1^\ngen(c,\psi\Lambda_q)\bigr)X^c+O(X^{3/4+o(1)}).
\end{array}
\end{equation*}
Then for $i\in\{1,2\}$, we have
\begin{equation}\label{eq:final_temp_1}
C_i(c,\psi,\Sigma)=\sum_q \mu(q)C_i(c,\psi,\Lambda_q)
=\sum_q \mu(q)C_i^\gen(c,\psi,\Lambda_q)+\sum_q \mu(q)C_i^\ngen(c,\psi,\Lambda_q).
\end{equation}
The values of $C_i^\gen(c,\psi,\Lambda_q)$ can be read off from Theorem \ref{thm:mainS4rings}, since $N_\psi^{(i)}(\phi_{\Lambda_q}^\gen;X)=N_{\Lambda_q}(\psi,X)$:
\begin{equation}\label{eq:final_temp_2}
C_2^\gen(1,\psi,\Lambda_q)=C_2^\gen(1,\psi,\Lambda_q)=0;\;\;
C_1^\gen(1,\psi,\Lambda_q)=\wt\psi(1)C_1(\Lambda_q);\;\;
C_1^\gen(5/6,\psi,\Lambda_q)=\wt\psi(5/6)C_{5/6}(\Lambda_q).
\end{equation}
Meanwhile, the values of $C_i^\ngen(c,\psi,\Lambda_q)$ can be read off from Theorem \ref{thm_red_points_count}:
\begin{equation*}
C_2^\ngen(1,\psi,\Lambda_q)=C_{1,q}^{\red'};\quad
C_2^\ngen(5/6,\psi,\Lambda_q)=0;\quad
C_1^\ngen(1,\psi,\Lambda_q)=C_{1,q}^{\red};\quad
C_1^\ngen(5/6,\psi,\Lambda_q)=C_{5/6,q}^{\red'}.
\end{equation*}
Now we write $\wt{N}_\Sigma(\psi,X)$ as the sum $N_\Sigma(\psi,X)+N^\red_\Sigma(\psi,X)$, where $N^\red_\Sigma(\psi,X)$ is the contribution from \'etale quartic algebras that are not $S_4$-rings. Since the contribution of $A_4$-fields and abelian fields is bounded by $O(X^{.778...})$, we obtain from Corollary \ref{cor:fields_from_rings_from_fields} that $N^\red_\Sigma(\psi,X)$ is, up to an error of $O(X^{.778...})$, equal to
\begin{equation*}
\Bigl(\sum_q\mu(q) C_2^\ngen(1,\psi,\Lambda_q)\Bigr)X\log X+\Bigl(\sum_q\mu(q) C_1^\ngen(1,\psi,\Lambda_q)\Bigr)X+\Bigl(\sum_q\mu(q) C_1^\ngen(5/6,\psi,\Lambda_q)\Bigr)X^{5/6}.
\end{equation*}
Therefore, combining with \eqref{eq:sieve_input}, \eqref{eq:final_temp_1}, and \eqref{eq:final_temp_2}, we obtain
\begin{equation*}
\begin{array}{rcl}
N_\Sigma(\psi,X)&=&\displaystyle\Bigl(\sum_q\mu(q)C_1(\Lambda_q)\Bigr)X+\Bigl(\sum_q\mu(q)C_{5/6}(\Lambda_q)\Bigr)X^{5/6}+O(X^{13/16+o(1)})
\\[.2in]
&=&C_1(\Sigma)\wt\psi(1)X+C_{5/6}(\Sigma)\wt\psi(5/6)X^{5/6}+O(X^{13/16+o(1)}),
\end{array}
\end{equation*}
as claimed by Theorem \ref{th_main_counting_result}. $\Box$

\subsection*{Acknowledgments}

We are very grateful to Will Sawin, Takashi Tanigcuhi, and Frank Thorne and for helpful conversations and useful comments on an earlier draft. The first-named author was supported by a Simons fellowship and an NSERC discovery grant.

\bibliographystyle{abbrv} \bibliography{Main}
\end{document}